\documentclass[11pt]{article}
\usepackage{amsmath}
\usepackage{amsfonts}
\usepackage{amssymb}
\usepackage{mathrsfs}
\usepackage{amsthm}
\usepackage{palatino}
\usepackage[margin=2.5cm, vmargin={1.5cm}]{geometry}

\usepackage{dcolumn}

\usepackage{times}
\usepackage{graphicx}
\usepackage{color}

\usepackage{pstricks}
\usepackage{pst-plot}

\newrgbcolor{LemonChiffon}{1.0 0.98 0.8}
\newrgbcolor{magenta}{1.0 0.3 0.4}
\newrgbcolor{mistyrose}{1.0 0.9 0.8}
\newrgbcolor{deeppink}{1.0 0.08 0.58}
\newrgbcolor{lightsalmon}{1.0 0.63 0.48}
\newrgbcolor{lightred}{0.7 0.0 0.0}
\newrgbcolor{lightblue}{0.8 0.8 1.0}
\newrgbcolor{indianred}{0.80  0.36  0.36}
\newrgbcolor{byellow}{0.93 0.76 0.3}
\newrgbcolor{light}{0.98 0.92 0.92}

\DeclareMathOperator{\cl}{Cl_2}

\renewcommand\footnotemark{}

\begin{document}

\title{Revisiting the saddle-point method of Perron}

\author{Cormac O'Sullivan\footnote{
\newline
Support for this project was provided by a PSC-CUNY Award, jointly funded by The Professional Staff Congress and The City University of New York.}}

\date{Feb 22, 2018}

\maketitle

\def\s#1#2{\langle \,#1 , #2 \,\rangle}

\def\H{{\mathbf{H}}}
\def\F{{\frak F}}
\def\C{{\mathbb C}}
\def\R{{\mathbb R}}
\def\Z{{\mathbb Z}}
\def\Q{{\mathbb Q}}
\def\N{{\mathbb N}}
\def\G{{\Gamma}}
\def\GH{{\G \backslash \H}}
\def\g{{\gamma}}
\def\L{{\Lambda}}
\def\ee{{\varepsilon}}
\def\K{{\mathcal K}}
\def\Re{\mathrm{Re}}
\def\Im{\mathrm{Im}}
\def\PSL{\mathrm{PSL}}
\def\SL{\mathrm{SL}}
\def\Vol{\operatorname{Vol}}
\def\lqs{\leqslant}
\def\gqs{\geqslant}
\def\sgn{\operatorname{sgn}}
\def\res{\operatornamewithlimits{Res}}
\def\li{\operatorname{Li_2}}
\def\lip{\operatorname{Li}'_2}
\def\pl{\operatorname{Li}}
\def\nb{{\mathcal B}}
\def\cc{{\mathcal C}}
\def\nd{{\mathcal D}}
\def\dd{\displaystyle}

\def\clp{\operatorname{Cl}'_2}
\def\clpp{\operatorname{Cl}''_2}
\def\farey{\mathscr F}

\newcommand{\stira}[2]{{\genfrac{[}{]}{0pt}{}{#1}{#2}}}
\newcommand{\stirb}[2]{{\genfrac{\{}{\}}{0pt}{}{#1}{#2}}}
\newcommand{\norm}[1]{\left\lVert #1 \right\rVert}


\newtheorem{theorem}{Theorem}[section]
\newtheorem{lemma}[theorem]{Lemma}
\newtheorem{prop}[theorem]{Proposition}
\newtheorem{conj}[theorem]{Conjecture}
\newtheorem{cor}[theorem]{Corollary}
\newtheorem{assume}[theorem]{Assumptions}

\newcounter{coundef}
\newtheorem{adef}[coundef]{Definition}

\newcounter{counrem}
\newtheorem{remark}[counrem]{Remark}

\renewcommand{\labelenumi}{(\roman{enumi})}
\newcommand{\spr}[2]{\sideset{}{_{#2}^{-1}}{\textstyle \prod}({#1})}
\newcommand{\spn}[2]{\sideset{}{_{#2}}{\textstyle \prod}({#1})}

\numberwithin{equation}{section}

\bibliographystyle{alpha}

\begin{abstract}
Perron's saddle-point method gives a way to find the complete asymptotic expansion of certain integrals that depend on a parameter going to infinity. We give two proofs of the key result. The first is a reworking of Perron's original proof, showing the clarity and simplicity that has been lost in some subsequent treatments. The second proof extends the approach of Olver which is based on Laplace's method. New results include more precise error terms and bounds for the expansion coefficients. We also treat Perron's original examples in greater detail and give a new application to the asymptotics of Sylvester waves.
\end{abstract}

\section{Introduction}
The main problem under consideration here is the accurate estimation of
\begin{equation} \label{i(n)}
    \int_{\cc} e^{N \cdot p(z)} q(z) \, dz
\end{equation}
as $N \to \infty$, where $p$ and $q$ are holomorphic functions and  integration is along a contour $\cc$. If the contour can be moved to pass through a saddle-point of $p(z)$ so that $\Re(p(z))$ achieves its maximum on $\cc$ there, then  the complete asymptotic expansion of \eqref{i(n)} may be given quite explicitly. This was established   one hundred years ago  by Perron in the groundbreaking paper \cite{Pe17}.

Unfortunately, this paper is now difficult to obtain. There seem to be two detailed accounts of the method that are more recent. Wong refers to {\em Perron's method} in  \cite[Part II, Sect. 5]{Wo89} and gives a statement and proof based on work of Wyman in \cite{Wy64}. These include an extra condition that does not appear in  \cite{Pe17}. The second account, by Olver in \cite[Thm.~6.1, p.~125]{Ol}, refers only to the {\em saddle-point method} and  does not include this extra condition. However it also does not include Perron's formula for the asymptotic expansion coefficients, nor give Perron's clear description of how the result is affected by the behavior of the contour $\cc$  near the saddle-point. Olver refers to \cite{Wy64} but his proof is different and more similar to Laplace's method.

To resolve these discrepancies, our first aim
is to produce a clear proof of the asymptotic expansion of \eqref{i(n)} based closely on Perron's original ideas. We see that the result may be  stated simply and is easy to apply.
  We also give a second proof that extends the work of Olver mentioned above. In two innovations,  the dependence of the error on $q(z)$ is made explicit, as required by our new application to the asymptotics of Sylvester waves in Section
\ref{syl}, and we show a bound for the expansion coefficients with Proposition \ref{albnd}. 

As a simple example of the asymptotics that Perron's method produces, we see in Section \ref{simex} that
\begin{equation*}
  \int_{1/2}^{3/2} e^{N(-z+\log z)}\, dz = \frac{\sqrt{2\pi}}{N^{1/2}e^N}\left(1+\frac{1}{12N}+\frac{1}{288N^2} -\frac{139}{51840 N^3}+O\left( \frac{1}{N^4}\right) \right)
\end{equation*}
as $N \to \infty$. Perron's original motivation was in finding the asymptotics  of the integral
\begin{equation}\label{vare}
   \int_{-\pi}^{\pi} \frac{e^{N i(z- \varepsilon\sin z)}}{1-\varepsilon\cos z} \,dz,
\end{equation}
which occurs in Kepler's theory when relating the true anomaly to the mean anomaly for a body orbiting in an ellipse with eccentricity $\varepsilon$. As described in \cite{Bu14}, the initial terms of the asymptotic expansion of \eqref{vare} had already been found by Jacobi, Cauchy and Debye, for example,  with difficult methods. Burkhardt in  \cite{Bu14} outlined a simpler approach and Perron was able to extend Burkhardt's ideas and make them rigorous. In \cite[Sect. 5]{Pe17} it is shown how to calculate as many terms as one wishes in the expansion of \eqref{vare} and several related integrals. We complete these examples in Section \ref{app}  by giving explicit formulas for all their expansion coefficients.

Perron's method has many other applications, for example to the asymptotics of special functions used in pure and applied mathematics \cite{Cop},\cite[Chap. 4]{Ol}, \cite{LPS2009b}, \cite{LP2011}, to statistics and probability \cite[Chap. 7]{Small}, and to results in combinatorics and number theory \cite[Chap. 6]{deB}, \cite[Sect. VIII]{Fl09}.
The author's interest in this area began with \cite{OS15,OS1}, where the method was key in obtaining the asymptotics of Rademacher's coefficients and disproving Rademacher's conjecture about them. The results described in Section
\ref{syl} on Sylvester waves are an extension of the work in \cite{OS1}.

\subsection{Main results}
The usual convention that the principal branch of $\log$ has arguments in $(-\pi,\pi]$ is used.
As in  \eqref{as} below, powers of nonzero complex numbers take the corresponding principal value $z^{\tau}:=e^{\tau\log(z)}$
for $\tau \in \C$. This convention will be in place throughout the paper, however in some cases we will specify different branches of the power.

Our contours of integration $\cc$ will lie in a bounded region of $\C$ and be parameterized by
  a continuous function   $c:[0,1]\to \C$ that has a continuous derivative except at a finite number of points.  For any appropriate $f$, integration along the corresponding contour $\cc$ is defined as $\int_\cc f(z)\, dz := \int_0^1 f(c(t))c'(t)\, dt$ in the normal way. 

The notation $f(z)=O(g(z))$, or equivalently $f(z) \ll g(z)$, means that there exists a $C$ so that $|f(z)|\lqs C\cdot g(z)$ for all $z$ in a specified range. The number $C$ is called the {\em implied constant}.


\SpecialCoor
\psset{griddots=5,subgriddiv=0,gridlabels=0pt}
\psset{xunit=1cm, yunit=1cm}
\psset{linewidth=1pt}
\psset{dotsize=4pt 0,dotstyle=*}

\begin{figure}[h]
\begin{center}
\begin{pspicture}(0,0.5)(9,3.5) 

\psset{arrowscale=2,arrowinset=0.5}
\pscircle[fillstyle=solid,linecolor=gray,fillcolor=mistyrose](4,2){1}

\pscurve[linecolor=black]{->}(0,1.5)(2,2.5)(4,2)(8,2.5)

\rput(7,1.9){$\cc$}
\rput(4,0.7){$\nb$}
\rput(4.3,2.25){$z_0$}

\psdots(0,1.5)(4,2)(8,2.5)

\end{pspicture}
\caption{Neighborhood $\nb$ and path of integration $\cc$}\label{basicfig}
\end{center}
\end{figure}

In our main results we make the following assumptions and definitions.

\begin{assume} \label{ma0}
We have $\nb$ a neighborhood of $z_0 \in \C$.  Let $\cc$ be  a contour as described above, with $z_0$  a point on it. Suppose $p(z)$ and $q(z)$ are holomorphic functions on a domain containing $\nb \cup \cc$.  We  assume $p(z)$ is not constant and  hence there must exist $\mu \in \Z_{\gqs 1}$ and $p_0 \in \C_{\neq 0}$ so that
\begin{equation}
    p(z)  =p(z_0)-p_0(z-z_0)^\mu(1-\phi(z)) \qquad  (z\in \nb) \label{f}
\end{equation}
with $\phi$  holomorphic on $\nb$ and $\phi(z_0)=0$.
Let $\omega_0:=\arg(p_0)$ and we will need the {\em steepest-descent angles}
\begin{equation}\label{bisec}
    \theta_\ell := -\frac{\omega_0}{\mu}+\frac{2\pi \ell}{\mu} \qquad (\ell \in \Z).
\end{equation}
For later results we  require $a \in \C$. We also assume that  $\nb,$ $\cc,$ $p(z),$  $q(z)$, $z_0$ and $a$  are independent of  $N>0$. Finally, let $K_q$ be a bound for $|q(z)|$ on $\nb \cup \cc$.
\end{assume}

The following is a slight restatement of Perron's key result in \cite[p. 202]{Pe17}. It may be compared with
\cite[Thm. 4, p. 105]{Wo89} and \cite[Thm. 6.1, p. 125]{Ol}.

\begin{theorem}\label{il} {\rm (Perron's method for a holomorphic integrand with  contour starting at a maximum.)}
Suppose that  Assumptions \ref{ma0} hold, with $\cc$  a contour from $z_0$ to $z_1$ in $\C$ where $z_0 \neq z_1$. Suppose  that
\begin{equation}\label{c1}
    \Re(p(z))<\Re(p(z_0)) \quad \text{for all} \quad z \in \cc, \ z\neq z_0.
\end{equation}
We may choose $k \in \Z$ so that the initial part of $\cc$ lies in the sector of angular width $2\pi/\mu$ about $z_0$ with bisecting angle $\theta_k$.
Then for every $S \in \Z_{\gqs 0}$, we have
\begin{equation} \label{wim}
    \int_\cc e^{N \cdot p(z)}  q(z) \, dz = e^{N \cdot p(z_0)} \left(\sum_{s=0}^{S-1}  \G\left(\frac{s+1}{\mu}\right) \frac{\alpha_s \cdot e^{2\pi i k (s+1)/\mu}}{N^{(s+1)/\mu}} + O\left(\frac{K_q}{N^{(S+1)/\mu}} \right)  \right)
\end{equation}
as $N \to \infty$ where  the implied constant in \eqref{wim} is independent of $N$ and $q$. The numbers $\alpha_s$
are given by
\begin{equation} \label{as}
    \alpha_s = \frac{1}{\mu \cdot  s!} p_0^{-(s+1)/\mu} \frac{d^s}{dz^s}\left\{q(z) \cdot \left( 1-\phi(z)\right)^{-(s+1)/\mu} \right\}_{z=z_0}.
\end{equation}.
\end{theorem}

To understand the geometry of the condition \eqref{c1}
 we first write
\begin{equation} \label{pexp}
    p(z)-p(z_0)=-\sum_{s=0}^\infty p_s(z-z_0)^{\mu+s} \qquad (z \in \nb).
\end{equation}
By Taylor's Theorem, for each $S$ there exists $K_{p,S}$ such that
\begin{equation} \label{tay}
    \left|p(z)-p(z_0)+\sum_{s=0}^{S-1} p_s(z-z_0)^{\mu+s}\right| \lqs K_{p,S} |z-z_0|^{\mu+S}
\end{equation}
for  all $z\in \nb$.
Write
\begin{equation} \label{letsx}
     p_s=|p_s| e^{i\omega_s} \quad \text{and} \quad z=z_0+r\cdot e^{i \theta}
\end{equation}
so that
 we obtain
\begin{equation} \label{cebx}
   \Re(p(z)-p(z_0))=-r^\mu \sum_{s=0}^\infty |p_s| r^s \cos \Bigl(\omega_s+(\mu+s)\theta\Bigr).
\end{equation}
Then \eqref{tay} and \eqref{cebx} imply that, for small $r$, $\Re( p(z)-p(z_0)) \approx -r^\mu |p_0|  \cos(\omega_0+\mu \theta)$. Hence, in a small neighborhood of $z_0$, the regions  where $\Re( p(z)-p(z_0))<0$ correspond approximately to $\mu$ sectors of angular width $\pi/\mu$. These `valleys' alternate with $\mu$ `hill' sectors, of the same size,  where $\Re( p(z)-p(z_0))>0$. The exact boundaries where $\Re( p(z)-p(z_0))=0$ will be differentiable curves, as we see in Section \ref{s3}. See Figure \ref{spokesfig} for an example with $\mu=3$.
\SpecialCoor
\psset{griddots=5,subgriddiv=0,gridlabels=0pt}
\psset{xunit=1cm, yunit=1cm}
\psset{linewidth=1pt}
\psset{dotsize=4pt 0,dotstyle=*}
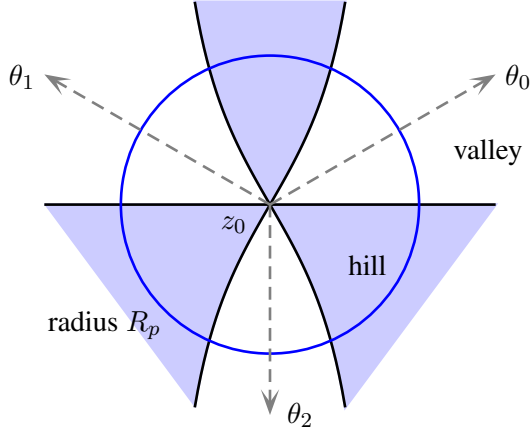
\begin{figure}[ht]
\begin{center}
\begin{pspicture}(-3,-2.8)(3,2.9) 

\psset{arrowscale=2,arrowinset=0.5}


\pscustom[linecolor=lightblue,fillstyle=solid,fillcolor=lightblue]{
\psplot[linecolor=red]{0}{1}{x 57.3 mul tan 1.732 mul}
\psplot[linecolor=red]{-1}{0}{0 x 57.3 mul tan 1.732 mul sub}}

\pscustom[linecolor=lightblue,fillstyle=solid,fillcolor=lightblue]{
\psplot[linecolor=red]{0}{1}{0 x 57.3 mul tan 1.732 mul sub}
\psline(3,0)(0,0)}

\pscustom[linecolor=lightblue,fillstyle=solid,fillcolor=lightblue]{
\psplot[linecolor=red]{-1}{0}{x 57.3 mul tan 1.732 mul}
\psline(0,0)(-3,0)}

\psplot[linecolor=black]{-1}{1}{x 57.3 mul tan 1.732 mul}
\psplot[linecolor=black]{-1}{1}{0 x 57.3 mul tan 1.732 mul sub}

\psline(-3,0)(3,0)

\pscircle[linecolor=blue](0,0){2}

\psline[linecolor=gray,linestyle=dashed]{->}(0,0)(0,-2.8)
\psline[linecolor=gray,linestyle=dashed]{->}(0,0)(3,1.732)
\psline[linecolor=gray,linestyle=dashed]{->}(0,0)(-3,1.732)

\rput(0.4,-2.8){$\theta_2$}
\rput(-3.3,1.7){$\theta_1$}
\rput(3.3,1.7){$\theta_0$}
\rput(2.9,0.7){valley}
\rput(1.3,-0.8){hill}
\rput(-0.48,-0.27){$z_0$}
\rput(-2.2,-1.6){radius $R_p$}

\end{pspicture}
\caption{Hills and valleys near $z_0=0$ for $p(z)=i(z-\sin z)$}\label{spokesfig}
\end{center}
\end{figure}
In Proposition \ref{basic} we show it is possible to  choose  $R_p>0$ and small enough so that these boundary curves behave nicely in the disk of  radius $R_p$ about $z_0$,  approximating $2\mu$ regularly spaced spokes in a wheel.

The bisecting lines of the valley sectors are clearly  given by $z_0+r e^{i \theta}$ for $r\gqs 0$ and $\theta$ satisfying  $\cos(\omega_0+\mu \theta)=1$. These bisecting angles are the $\theta_\ell$ defined in \eqref{bisec} and  correspond to the directions of greatest decrease (steepest descent) of $\Re( p(z)-p(z_0))$.

The condition \eqref{c1} means that the initial part of $\cc$ must lie in one of the valley regions.  To specify which one, we use the fact that the part of this region within a distance $R_p$ from $z_0$
 must lie inside the sector of  angular width $2\pi/\mu$ about $z_0$ with bisecting angle $\theta_k$ for some $k\in \Z$. For the details of this see Section \ref{s3}.

  The proofs of Theorem \ref{il} we give in Sections \ref{fp} and \ref{spf} rely on the important simplification of Perron stated next and proved in Section \ref{s3}.

\begin{prop} \label{ilz}
Suppose  all the assumptions of Theorem \ref{il} are true. Let $b$ be the point on the bisecting line with angle $\theta_k$ that is a distance $R_p$ from $z_0$. Then there exists $\varepsilon>0$ so that
\begin{equation} \label{wimz}
    \int_\cc e^{N \cdot p(z)}  q(z) \, dz = e^{N \cdot p(z_0)} \left(\int_{z_0}^{b} e^{N (p(z)-p(z_0))}  q(z) \, dz  + O\left(K_q e^{-\varepsilon N} \right)  \right)
\end{equation}
as $N \to \infty$ where $\varepsilon$ and the implied constant in \eqref{wimz} are independent of $N$ and $q$.
\end{prop}


The point $b$ is shown in Figure \ref{newpathfig}. It is clear from  Proposition \ref{ilz} that most details of the contour $\cc$ are irrelevant for our asymptotic results; we only need to know which sector the contour starts off in.

As a simple corollary to Theorem \ref{il}, the next result is obtained by breaking the
contour of integration into  $\int_\cc=\int_{z_0}^{z_2}-\int_{z_0}^{z_1}$. This may also be compared with Theorem 1 of \cite{LPS2009b}.

\begin{cor} \label{il2} {\rm (Perron's method for a holomorphic integrand with contour passing  through  a maximum.)}
Suppose Assumptions \ref{ma0} hold. Let $\cc$ be a contour starting at $z_1$, passing through $z_0$ and ending at  $z_2$,  with these three points all distinct.  Suppose that
\begin{equation}\label{c1x}
   \Re(p(z))<\Re(p(z_0)) \quad \text{for all} \quad z \in \cc, \ z\neq z_0.
\end{equation}
Let $\cc$ approach $z_0$ in the sector of angular width $2\pi/\mu$ about $z_0$ with bisecting angle $\theta_{k_1}$ and leave $z_0$ in a sector of the same size with bisecting angle $\theta_{k_2}$.
Then for every $S \in \Z_{\gqs 0}$, we have
\begin{equation} \label{wimab}
    \int_\cc e^{N \cdot p(z)}  q(z) \, dz = e^{N \cdot p(z_0)} \left(\sum_{s=0}^{S-1}  \G\left(\frac{s+1}{\mu}\right) \frac{\alpha_s ( e^{2\pi i k_2 (s+1)/\mu}- e^{2\pi i k_1 (s+1)/\mu})}{N^{(s+1)/\mu}} + O\left(\frac{K_q}{N^{(S+1)/\mu}} \right)  \right)
\end{equation}
as $N \to \infty$ where  the implied constant  is independent of $N$ and $q$.
 The numbers $\alpha_{s}$ are given by \eqref{as}.
\end{cor}

We will see generalizations of these results in Section \ref{general}. In Section \ref{formul}, more explicit formulas for the numbers $\alpha_{s}$ are given.

Prior to \cite{Bu14} and \cite{Pe17}, different techniques to estimate integrals by moving the path of integration to a saddle-point were pioneered by Cauchy, Stokes, Riemann, Nekrasov, Kelvin and Debye.  See  for example \cite{Ol70}, \cite[pp. 104-105]{Ol}, \cite{Pet97} and \cite{Tem13} where their contributions   are described. These techniques include the method of steepest descent, and an advantage of Corollary \ref{il2} is that it does not require computing steepest descent paths.

\subsection{Burkhardt's heuristic}
Before proving the above results, we give Burkhardt's heuristic and show how the form of \eqref{wimab} arises.
Suppose $p'(z_0)=0$ and $p''(z_0)<0$. For simplicity we take $\cc=[-1,1]$ and $z_0=0$. Expanding $p(z)$ as in \eqref{pexp} with $p(z)=p(0)-p_0z^2-p_1z^3- \cdots$ and $q(z)=q_0+q_1 z+\cdots$ yields
\begin{equation*}
  \int_{-1}^1 e^{N \cdot p(z)}q(z)\, dz=  e^{N  \cdot p(0)}\int_{-1}^1 e^{-N p_0 z^2}
  e^{-N z^2(p_1 z+p_2 z^2+ \cdots)}
  (q_0+q_1 z+\cdots)\, dz
\end{equation*}
where we may write
\begin{equation*}
  e^{-N z^2(p_1 z+p_2 z^2+ \cdots)} = 1-N z^2(p_1 z+p_2 z^2+ \cdots)+\frac{(Nz^2)^2}{2}(p_1 z+p_2 z^2+ \cdots)^2+\cdots.
\end{equation*}
Since $p_0>0$ and $N>0$, the term $e^{-N p_0 z^2}$ will have exponential decay and so extending the path of integration to $\R$ should not affect the result.
Let $w=N p_0 z^2$ to obtain
\begin{equation} \label{burk}
    e^{N  \cdot p(0)}\int_{-\infty}^{\infty} e^{-w}
  \left(1-\frac{w}{p_0}(p_1 z+p_2 z^2+ \cdots)+\frac{w^2}{2p_0^2}(p_1 z+p_2 z^2+ \cdots)^2+\cdots \right)
  (q_0+q_1 z+\cdots)\, dz.
\end{equation}
By symmetry, the contributions from the odd powers of $z$ will cancel. From the $z^0$ term of \eqref{burk} we get the first term of the asymptotic expansion:
\begin{equation} \label{brk1}
  2 e^{N \cdot  p(0)}\int_{0}^{\infty} e^{-w}q_0 \frac{dw}{2(N p_0)^{1/2} w^{1/2}} = e^{N \cdot  p(0)}\frac{\G(1/2) q_0}{(N p_0)^{1/2}}.
\end{equation}
From the $z^2$ term of \eqref{burk} we get the next term of the expansion:
\begin{multline} \label{brk2}
  2 e^{N  \cdot p(0)}\int_{0}^{\infty} e^{-w}
  \left( q_2-\frac{w(p_1q_1+p_2q_0)}{p_0}+\frac{w^2 p_1^2 q_0}{2p_0^2} \right)z^2
  \frac{dw}{2(N  p_0)^{1/2} w^{1/2}} \\
  =  \frac{e^{N  \cdot p(0)}}{(N p_0)^{3/2}}\int_{0}^{\infty} e^{-w}
  \left( w q_2-\frac{w^2(p_1q_1+p_2q_0)}{p_0}+\frac{w^3 p_1^2 q_0}{2p_0^2} \right)
  \frac{dw}{ w^{1/2}} \\
   =  \frac{e^{N \cdot  p(0)}}{(N p_0)^{3/2}}
  \left( \G(3/2) q_2-\frac{\G(5/2)(p_1q_1+p_2q_0)}{p_0}+\frac{\G(7/2) p_1^2 q_0}{2p_0^2} \right).
\end{multline}
The formulas \eqref{brk1} and \eqref{brk2} will reappear in Section \ref{formul}.

\section{Preliminary results} \label{s3}
 This section is an elaboration of the paragraph in \cite{Pe17} before equation $(11)$ and gives a detailed description of $p(z)$ for $z$ near $z_0$.
\begin{prop} \label{basic}
Suppose $p(z)$ is holomorphic in a neighborhood $\nb$ of $z_0$.  As in Assumptions \ref{ma0}, we  assume $p(z)$ is not constant and  hence there must exist $\mu \in \Z_{\gqs 1}$ and $p_0 \in \C_{\neq 0}$ so that
\begin{equation}
    p(z)  =p(z_0)-p_0(z-z_0)^\mu(1-\phi(z)) \qquad  (z\in \nb) \label{fzx}
\end{equation}
with $\phi$  holomorphic on $\nb$ and $\phi(z_0)=0$. Then there exists $R_p>0$ so that the closed disk centered at $z_0$ of radius $R_p$  is contained in $\nb$ and we have  the following additional properties.
\begin{enumerate}
  \item All solutions to $\Re\bigl(p(z_0+r e^{i \theta})-p(z_0)\bigr)/r^\mu=0$ for $r \in [0,R_p]$ have the form $(r,\theta)=(r,f_\ell(r))$ for functions $f_\ell(r)$ with $\ell \in \Z$.
  \item These functions $f_\ell(r)$ are all defined on an interval containing $[0,R_p]$ and are differentiable.
  \item We have
  \begin{equation} \label{delel}
   f_\ell(0) = \delta_\ell \quad \text{for} \quad \delta_\ell:= -\frac{\omega_0}{\mu}+\frac{\pi(\ell+1/2)}{\mu}.
  \end{equation}
  \item Also $\left| f_\ell(r)-\delta_\ell\right| \lqs \pi/(4\mu)$ for $r \in [0,R_p]$.
\end{enumerate}
\end{prop}
\begin{proof}
Set $H(r,\theta):= -\Re\bigl(p(z_0+r e^{i \theta})-p(z_0)\bigr)/r^\mu$. By \eqref{cebx}
\begin{equation*}
    H(r,\theta) = \sum_{s=0}^\infty |p_s| r^s \cos \Bigl(\omega_s+(\mu+s)\theta\Bigr)
\end{equation*}
and so $H(0,\theta) = |p_0| \cos(\omega_0+\mu\theta)$. Then the  solutions to $H(0,\theta)=0$ are $\theta=\delta_\ell$ for $\ell \in \Z$ with $\delta_\ell$ defined in \eqref{delel}.

For $(r,\theta)$ in a neighborhood of $(0,\delta_\ell)$ the partial derivatives of $H(r,\theta)$ exist and are continuous. Also
\begin{equation*}
    \left.\frac{\partial H}{\partial \theta}\right|_{(r,\theta)=(0,\delta_\ell)} = -|p_0|\mu\sin(\omega_0+\mu\delta_\ell) = (-1)^{\ell+1}|p_0|\mu \neq 0.
\end{equation*}
Therefore, by the Implicit Function Theorem, all the solutions to $H(r,\theta)=0$ for $(r,\theta)$ in some neighborhood of $(0,\delta_\ell)$ take the form $(r,\theta)=(r,f_\ell(r))$ for differentiable functions $f_\ell$.
Note that $H(r,\theta+2\pi)=H(r,\theta)$ so that, for all $ \ell\in \Z$,
\begin{equation} \label{modf}
    f_{\ell+2\mu}(r) = f_\ell(r)+2\pi.
\end{equation}
We choose $R_p>0$ small enough so that the interval $[0,R_p]$ is contained in the above neighborhoods for all  $\ell \in \Z$. By \eqref{modf}, this choice involves only $2\mu$ conditions.
 We have proved parts (i), (ii) and (iii).

Suppose $\epsilon>0$ is given. Since $f_\ell(r)$ is continuous at $r=0$ we may decrease $R_p$ again, if necessary, to ensure that $|f_\ell(r) - f_\ell(0)|\lqs \epsilon$ for $r\in [0,R_p]$. We do this for each $\ell \bmod 2\mu$ and with $\epsilon = \pi/(4\mu)$. This proves part (iv).
\end{proof}

\begin{cor} \label{basic2}
Suppose all the assumptions of Proposition \ref{basic} hold. Then
\begin{equation}\label{repl>b}
    f_{2\ell-1}(r)<\theta_\ell<f_{2\ell}(r) \quad \text{for all} \quad r\in [0,R_p], \ \ell \in \Z.
\end{equation}
Also
\begin{equation}\label{repl>}
   \Re( p(z_0+r e^{i \theta_\ell})-p(z_0))<0 \quad \text{for all} \quad r\in (0,R_p], \ \ell \in \Z.
\end{equation}
Inequalities \eqref{repl>b} and \eqref{repl>} are special cases of the following.
For every $r\in (0,R_p]$ we have
\begin{equation} \label{rep>}
   \Re( p(z_0+r e^{i \theta})-p(z_0))<0
\end{equation}
if and only if $\theta$ satisfies $f_{2\ell-1}(r)<\theta<f_{2\ell}(r)$ for some $\ell \in \Z$.
\end{cor}
\begin{proof}
By Proposition \ref{basic}, part (iii) we have
\begin{equation*}
   f_{2\ell-1}(0) +\frac{\pi}{2\mu} = \theta_\ell =  f_{2\ell}(0) -\frac{\pi}{2\mu}.
\end{equation*}
Hence, with part (iv), it is clear that \eqref{repl>b} holds. Therefore $\Re( p(z_0+r e^{i \theta_\ell})-p(z_0))$ does not change sign for $r\in (0,R_p]$. Since
\begin{equation*}
  \Re( p(z_0+r e^{i \theta_\ell})-p(z_0)) \approx -r^\mu |p_0|  \cos(\omega_0+\mu \theta_\ell) = -r^\mu |p_0| <0
\end{equation*}
for small $r$ we obtain \eqref{repl>}. Similarly, along the directions of steepest ascent,
\begin{equation}\label{repl<}
   \Re( p(z_0+r e^{i (\theta_\ell+\pi/\mu)})-p(z_0))>0 \quad \text{for all} \quad r\in (0,R_p], \ \ell \in \Z.
\end{equation}

For fixed $r\in (0,R_p]$, consider $\Re( p(z_0+r e^{i \theta})-p(z_0))$ as a continuous function of $\theta$ with zeros only at $\theta=f_\ell(r)$ for $\ell \in \Z$. Therefore $\Re( p(z_0+r e^{i \theta})-p(z_0))$ is always positive or always negative for $f_{2\ell-1}(r)<\theta<f_{2\ell}(r)$. By \eqref{repl>b} and \eqref{repl>} it must be negative. Similarly, with \eqref{repl<}, it must be positive for $f_{2\ell}(r)<\theta<f_{2\ell+1}(r)$.
\end{proof}


\SpecialCoor
\psset{griddots=5,subgriddiv=0,gridlabels=0pt}
\psset{xunit=1.6cm, yunit=1.6cm, runit=1.6cm}
\psset{linewidth=1pt}
\psset{dotsize=5pt 0,dotstyle=*}

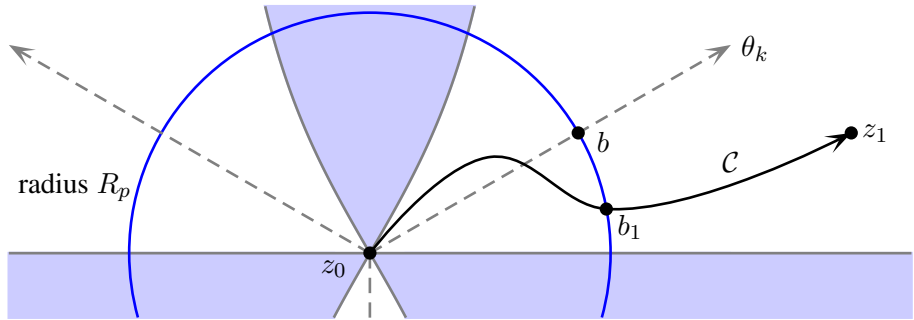
\begin{figure}[ht]
\begin{center}
\begin{pspicture}(-3,-0.5)(4.5,2) 

\psset{arrowscale=2,arrowinset=0.5}

\pscustom[linecolor=lightblue,fillstyle=solid,fillcolor=lightblue]{
\psplot[linecolor=red]{0}{0.87}{x 57.3 mul tan 1.732 mul}
\psplot[linecolor=red]{-0.87}{0}{0 x 57.3 mul tan 1.732 mul sub}}

\pscustom[linecolor=lightblue,fillstyle=solid,fillcolor=lightblue]{
\psplot[linecolor=red]{0}{0.3}{0 x 57.3 mul tan 1.732 mul sub}
\psline(3,0)(0,0)}

\pscustom[linecolor=lightblue,fillstyle=solid,fillcolor=lightblue]{
\psplot[linecolor=red]{-0.3}{0}{x 57.3 mul tan 1.732 mul}
\psline(0,0)(-3,0)}

\pspolygon[linecolor=lightblue,fillstyle=solid,fillcolor=lightblue](-3,0)(-0.3,0)(-0.3,-0.536)(-3,-0.536)
\pspolygon[linecolor=lightblue,fillstyle=solid,fillcolor=lightblue](4.5,0)(0.3,0)(0.3,-0.536)(4.5,-0.536)

\psplot[linecolor=gray]{-0.3}{0.87}{x 57.3 mul tan 1.732 mul}
\psplot[linecolor=gray]{-0.87}{0.3}{0 x 57.3 mul tan 1.732 mul sub}

\psline[linecolor=gray](-3,0)(4.5,0)

\psarc[linecolor=blue](0,0){2}{-15.5}{195.5}

\psline[linecolor=gray,linestyle=dashed](0,0)(0,-0.536)
\psline[linecolor=gray,linestyle=dashed]{->}(0,0)(3,1.732)
\psline[linecolor=gray,linestyle=dashed]{->}(0,0)(-3,1.732)

\pscurve[linecolor=black]{->}(0,0)(1,0.8)(1.965, 0.3675)(4,1)

\psdots(1.732,1)(1.965, 0.3675)(4,1)(0,0)
\rput(3.2,1.7){$\theta_k$}
\rput(3,0.76){$\cc$}
\rput(1.94,0.93){$b$}
\rput(2.16,0.22){$b_1$}
\rput(4.2,1){$z_1$}
\rput(-0.3,-0.13){$z_0$}
\rput(-2.45,0.55){radius $R_p$}

\end{pspicture}
\caption{Replacing $\cc$ by the line from $z_0$ to $b$}
\label{newpathfig}
\end{center}
\end{figure}

\begin{proof}[Proof of Proposition \ref{ilz}] If the contour $\cc$ is not contained in the disk of radius $R_p$ about $z_0$ then let $b_1$ be the first point of $\cc$ that is a distance $R_p$ from $z_0$, as shown in Figure \ref{newpathfig}. Let $\cc'$ be the contour from $b$ to $z_1$ that follows the circular arc about $z_0$ from $b$ to $b_1$. From $b_1$ the contour now follows $\cc$ to $z_1$. (If $\cc$ is  contained in the disk of radius $R_p$ about $z_0$ then $\cc'$ could move from $b$ to a point $b_0$  on the line between $z_0$ and $b$ that is the same distance as $z_1$ from $z_0$. It then follows the circular arc about $z_0$ from $b_0$ to $z_1$.)

Since the integrand is holomorphic, Cauchy's Theorem tells us that
\begin{equation*}
    \int_\cc e^{N \cdot p(z)}  q(z) \, dz = \left( \int_{z_0}^b + \int_{\cc'}\right) e^{N \cdot p(z)}  q(z) \, dz.
\end{equation*}
It is clear from Corollary \ref{basic2} and \eqref{c1} that $\Re(p(z)-p(z_0))<0$ for $z\in \cc'$. Hence there exists  $\varepsilon>0$, depending only on   $\cc$, $p(z)$ and $R_p$, such that $\Re(p(z)-p(z_0))\lqs -\varepsilon$ for all $z\in \cc'$.  Therefore
\begin{equation} \label{fft2}
    \left| \int_{\cc'}  e^{N (p(z)-p(z_0))} q(z)\, dz \right| \lqs e^{-\varepsilon N}  \int_{\cc'} | q(z)|\, |dz|
    \lqs K_q |\cc'|  e^{-\varepsilon N}
\end{equation}
where $|\cc'|$ is the length of $\cc'$ which is less than $R_p+R_p(\pi/\mu)+|\cc|$. This completes the proof of Proposition \ref{ilz}.
\end{proof}

Therefore Perron shows us that in finding the asymptotic expansion of \eqref{i(n)}, we may replace $\cc$ by the line from $z_0$ to $b$ as shown in Figure \ref{newpathfig}. This important step is emphasized in \cite{LPS2009b}. Theorem 4 on p. 105 of \cite{Wo89} (based on the corresponding result of \cite{Wy64}) is similar to Theorem \ref{il} but has the extra condition that there exists $\delta>0$ so that $|\arg(p(z_0)-p(z))|\lqs \pi/2-\delta$ for all $z\in \cc$. This condition seems to be caused by missing the step of Proposition \ref{ilz}. Olver  also comments in \cite{Ol70} that this condition  is unnecessary. (There are two further unnecessary conditions in \cite{Wo89}: that the initial part of $\cc$ may be deformed into a straight line and that the path $\cc$ leaves $z_0$ at a well-defined angle.)

\section{First proof of Theorem \ref{il}} \label{fp}

This proof of Theorem \ref{il} is based closely on Perron's original in \cite{Pe17} though including more detail. We follow Wyman \cite{Wy64} in bounding $P_s(w)$ in Lemma \ref{l1} using Cauchy's inequality. We also depart from Perron by bounding $Q_s(z)$ in Lemma \ref{pq} using the integral form of the remainder from Taylor's Theorem.


\SpecialCoor
\psset{griddots=5,subgriddiv=0,gridlabels=0pt}
\psset{xunit=1cm, yunit=1cm, runit=1cm}
\psset{linewidth=1pt}
\psset{dotsize=5pt 0,dotstyle=*}

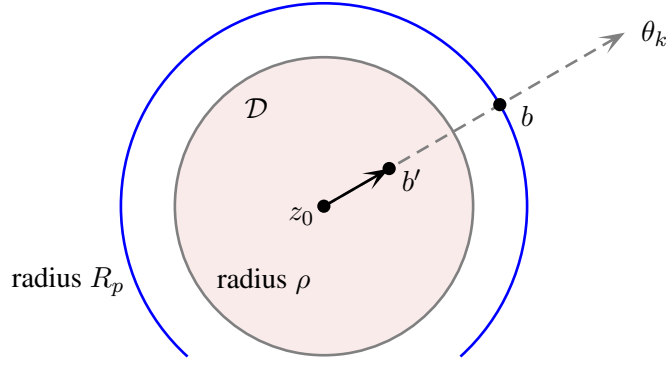
\begin{figure}[ht]
\begin{center}
\begin{pspicture}(-5,-2)(5,3) 

\psset{arrowscale=2,arrowinset=0.5}

\psarc[linecolor=blue](0,0){2.7}{-47.7946}{227.795}
\pscircle[fillstyle=solid,linecolor=gray,fillcolor=light](0,0){2}

\psline[linecolor=gray,linestyle=dashed]{->}(0,0)(4,2.3094)
\psline{->}(0,0)(0.866, 0.5)


\psdots(2.33827, 1.35)(0.866, 0.5)(0,0)
\rput(4.4,2.3){$\theta_k$}
\rput(-0.9,1.3){$\nd$}
\rput(2.7,1.2){$b$}
\rput(1.166,0.35){$b'$}
\rput(-0.3,-0.13){$z_0$}
\rput(-3.4,-1){radius $R_p$}
\rput(-0.8,-1){radius $\rho$}

\end{pspicture}
\caption{The line from $z_0$ to $b'$ is the new path of integration}
\label{pf1fig}
\end{center}
\end{figure}

\begin{proof}[Proof of Theorem \ref{il}]
Let
\begin{equation*}
    \nd := \{z\in \C \ : \ |z-z_0| \lqs \rho\}
\end{equation*}
for  $\rho =R_p$ initially.
Since $\phi(z_0)=0$, there  exists $K_\phi>0$ such that
\begin{equation}\label{kphi}
    |\phi(z)| \lqs K_\phi |z-z_0| \quad \text{for all} \quad z \in \nd.
\end{equation}
Looking ahead to Lemma \ref{pq}, we decrease $\rho$, if necessary, to ensure that
\begin{equation}\label{krho}
    0 < \rho \lqs 1/(2K_\phi).
\end{equation}

By Proposition \ref{ilz} we only need to estimate the integral
\begin{equation*}
    \int_{z_0}^{b} e^{N (p(z)-p(z_0))}  q(z) \, dz
\end{equation*}
where $b$ is on the bisecting line with angle $\theta_k$ and a distance $R_p$ from $z_0$. It is convenient to change the end point to $b'$, on the same bisecting line and a distance $\rho/2$ from $z_0$. See Figure \ref{pf1fig}. By \eqref{repl>} there exists $\varepsilon'>0$ such that $\Re(p(z)-p(z_0))\lqs -\varepsilon'$ for $z$ on the line between $b'$ and $b$. Hence
\begin{equation} \label{cby}
    \int_{z_0}^{b'} e^{N (p(z)-p(z_0))}  q(z) \, dz = \int_{z_0}^{b} e^{N (p(z)-p(z_0))}  q(z) \, dz + O\left(K_q e^{-\varepsilon' N} \right).
\end{equation}

For any $w \in \C$ we have the Taylor expansion
\begin{equation*}
    q(z)e^{w \phi(z)}=\sum_{s=0}^\infty P_s(w) (z-z_0)^s \qquad (z \in \nd).
\end{equation*}
Since
\begin{equation*}
    q(z)e^{w \phi(z)}=q(z)\left(1+w \phi(z)+ w^2 \phi(z)^2/2!+ \cdots \right)
\end{equation*}
and $\phi(z_0)=0$, it follows that $P_s(w)$ is a polynomial and
\begin{equation}\label{polyP}
  P_s(w) = \sum_{\ell=0}^s c_{s,\ell} \cdot w^\ell
\end{equation}
where $c_{s,\ell}$ is the coefficient of $(z-z_0)^s$ in the Taylor expansion of $q(z)\phi(z)^\ell/\ell!$ about $z_0$. The following bound for $P_s(w)$ will be needed for the proof of Proposition \ref{min4}.

\begin{lemma} \label{l1}
For all $w\in \C$
\begin{equation*}
    |P_s(w)| \lqs  K_q e^{K_\phi}(|w|^s + \rho^{-s}).
\end{equation*}
\end{lemma}
\begin{proof}
Starting with Cauchy's inequality, \cite[p. 120]{Al}, we find that for every $r$ with $0<r\lqs \rho$,
\begin{align}
    |P_s(w)| & \lqs r^{-s} \max_{|z-z_0|=r} |q(z)e^{w \phi(z)}| \notag\\
    & \lqs K_q r^{-s}  \max_{|z-z_0|=r} e^{\Re(w \phi(z))} \notag \\
    & \lqs K_q r^{-s}   e^{K_\phi |w| r}. \label{pia}
\end{align}
If $|w| \lqs 1/\rho$ then letting $r=\rho$ in \eqref{pia} shows $|P_s(w)| \lqs K_q e^{K_\phi} \rho^{-s}$. If $|w| \gqs 1/\rho$ then letting $r=1/|w|$ in \eqref{pia} shows $|P_s(w)| \lqs K_q e^{K_\phi} |w|^s$.
\end{proof}

Now we take
\begin{equation}\label{ww}
    w=N p_0 (z-z_0)^\mu.
\end{equation}
It is an easy exercise to check that $w\gqs 0$ when $z$ is on the line between $z_0$ and $b'$. For these $z$ values
\begin{equation}\label{wwt}
  w^{1/\mu}=(N |p_0|)^{1/\mu}|z-z_0|.
\end{equation}

\begin{lemma} \label{pq}
With $w$ given by \eqref{ww}, and $z$ on the line between $z_0$ and $b'$, we have
\begin{equation} \label{gf}
    e^{N (p(z)-p(z_0))} q(z) = \sum_{s=0}^{S-1} e^{-w} P_s(w) (z-z_0)^{s} + Q_S(z)
\end{equation}
where
\begin{equation} \label{gf3}
    \left| Q_S(z) \right| \lqs \frac{2K_q}{\rho^{S}} |z-z_0|^{S} e^{-w/2}.
\end{equation}
\end{lemma}
\begin{proof}
By Taylor's Theorem, see \cite[pp. 125-126]{Al},
\begin{equation*}
    q(z)e^{w \phi(z)}=\sum_{s=0}^{S-1} P_s(w) (z-z_0)^s+\frac{(z-z_0)^S}{2\pi i}\int_\g \frac{q(\tau)e^{w \phi(\tau)}}{(\tau-z_0)^S(\tau-z)}\, d\tau
\end{equation*}
where $\g$ is the positively oriented circle of radius $\rho$ about $z_0$. For $\tau \in \g$ we have $|q(\tau)e^{w \phi(\tau)}| \lqs K_q  e^{K_\phi w \rho}$. Also $|\tau-z|\gqs \rho/2$ since $|z-z_0| \lqs \rho/2$ by our choice of $b'$. The identity
\begin{equation*}
    e^{N (p(z)-p(z_0))} q(z) =  e^{-w} \cdot q(z)e^{w \phi(z)}
\end{equation*}
proves \eqref{gf} with
\begin{equation*} 
    \left| Q_S(z) \right| \lqs \frac{2K_q}{\rho^{S}} |z-z_0|^{S} e^{-w +K_\phi w \rho}.
\end{equation*}
The inequality  \eqref{krho} implies $\exp(-w +K_\phi w \rho)  \lqs \exp(-w/2 )$ and we obtain \eqref{gf3}.
\end{proof}

With  Proposition \ref{ilz}, \eqref{cby} and Lemma \ref{pq}
we may write
\begin{equation} \label{thre}
    \int_\cc e^{N \cdot p(z)} q(z) \, dz  = e^{N \cdot  p(z_0)}\left( \sum_{s=0}^{S-1} I_s(N) +  \int_{z_0}^{b'} Q_S(z) \, dz
    +   O\left(K_q e^{-\varepsilon N}\right)
    \right)
\end{equation}
for
\begin{equation} \label{isl}
     I_s(N) := \int_{z_0}^{b'} e^{-w} P_s(w) (z-z_0)^{s} \, dz
\end{equation}
(with $w$ given  by \eqref{ww}) and where $\varepsilon>0$ is independent of $N$ and $q$.

\begin{lemma} \label{min}
We have
\begin{equation*}
    \int_{z_0}^{b'} Q_S(z) \, dz =O\left(\frac{K_q}{N^{(S+1)/\mu}}\right).
\end{equation*}
\end{lemma}
\begin{proof}
The absolute value of the left side is
\begin{align}
     \left| \int_0^{\rho/2} Q_S(t e^{i \theta_k}+z_0) e^{i \theta_k} \, dt \right|
    & \lqs \int_0^{\rho/2} \left| Q_S(t e^{i \theta_k}+z_0) \right| \, dt \notag\\
    & \lqs \frac{2K_q}{\rho^{S}}  \int_0^{\rho/2} \exp(- N |p_0| t^\mu/2) \cdot t^{S} \, dt. \label{icefr}
\end{align}
We used  inequality \eqref{gf3} in \eqref{icefr} and that $w=N |p_0| t^\mu $ when $z=t e^{i \theta_k}+z_0$.
With the change of variables $u= N |p_0| t^\mu/2$ and extending the range of integration to $\infty$ we obtain
\begin{equation*} \label{fft3}
    \left| \int_{z_0}^{b'} Q_S(z) \, dz \right| \lqs \left[\frac{2\G((S+1)/\mu)}{\mu \cdot \rho^{S} ( |p_0|/2)^{(S+1)/\mu}} \right]
    \frac{K_q}{N^{(S+1)/\mu}}. \qedhere
\end{equation*}
\end{proof}

Combining the errors  from \eqref{thre} and Lemma  \ref{min} shows
\begin{equation} \label{thre2}
    \int_\cc e^{N \cdot p(z)} q(z) \, dz = e^{N \cdot p(z_0)}\left( \sum_{s=0}^{S-1} I_s(N) +  O\left(\frac{K_q}{N^{(S+1)/\mu}}\right)
    \right)
\end{equation}
for an implied constant independent of $N$ and $q$.

\begin{lemma} \label{min3}
We have
\begin{equation*}
I_s(N) = \frac{e^{2\pi i k (s+1)/\mu}}{\mu \cdot(N p_0)^{(s+1)/\mu}} \int_0^{N|p_0|(\rho/2)^\mu} e^{-w} P_s(w) w^{(s+1)/\mu-1}\, dw.
\end{equation*}
\end{lemma}
\begin{proof}
Recall \eqref{ww}. First we claim that
\begin{equation} \label{z-a}
    z-z_0 = w^{1/\mu} (N p_0)^{-1/\mu} e^{2\pi i k/\mu}
\end{equation}
for $z$ on the line between $z_0$ and $b'$.  This follows from the definitions
\begin{equation*}
  p_0=|p_0|e^{i \omega_0}, \qquad \theta_k = -\frac{\omega_0}{\mu}+\frac{2\pi k}{\mu}, \qquad z-z_0=|z-z_0| e^{i \theta_k}
\end{equation*}
and the relation \eqref{wwt}.
The proof is completed by using \eqref{z-a} in \eqref{isl} to change the variable of integration to $w$. 
\end{proof}

\begin{prop} \label{min4}
There exists $\varepsilon''>0$ so that
\begin{multline} \label{int}
 \int_0^{N|p_0|(\rho/2)^\mu} e^{-w} P_s(w) w^{(s+1)/\mu-1}\, dw \\
 = \frac{\G((s+1)/\mu)}{s!}\frac{d^s}{dz^s}\left\{q(z) \left( 1-\phi(z)\right)^{-(s+1)/\mu)} \right\}_{z=z_0} + O(K_q e^{-\varepsilon'' N}).
\end{multline}
\end{prop}
\begin{proof}
Put $T:=N|p_0|(\rho/2)^\mu$  and write the integral in \eqref{int} as $\int_0^T=\int_0^\infty-\int_T^\infty$.
Employing Lemma \ref{l1}, we find
\begin{align}
    \left|\int_{T}^{\infty} e^{-w} P_s(w) w^{(s+1)/\mu-1}\, dw \right|
    & \lqs  K_q e^{K_\phi} \int_T^{\infty} e^{-w}(w^s+\rho^{-s}) w^{(s+1)/\mu-1} \,dw \notag\\
    & =  K_q e^{K_\phi}\left(\int_T^{\infty} e^{-w} w^{d-1} \,dw + \rho^{-s}\int_T^{\infty} e^{-w} w^{d'-1} \,dw \right) \label{rra}
\end{align}
for $d:=(s+1)/\mu+s$ and $d':=(s+1)/\mu$.
The estimate
\begin{equation}\label{gam}
    \int_T^{\infty} e^{-w} w^{d-1} \,dw \lqs 2^d \G(d) e^{-T/2} \qquad (T\gqs 0, d>0)
\end{equation}
follows from bounding $e^{-w}$ in the integrand by $e^{-T/2}e^{-w/2}$.
(More accurate estimates of the incomplete Gamma function are possible; see for example \cite[Eq. (2.02), p. 110]{Ol}.)
Hence \eqref{rra} is bounded by
\begin{equation*}
      K_q e^{K_\phi} \Bigl(2^d \G(d)+\rho^{-s}2^{d'} \G(d')\Bigr) e^{-T/2}.
\end{equation*}
We have shown that
\begin{equation*} \label{fft4}
    \left|\int_{T}^{\infty}  \right| = O( K_q  e^{-\varepsilon'' N})
\end{equation*}
for $\varepsilon''= |p_0|(\rho/2)^\mu/2$ and an implied constant independent of $N$ and $q$.

Lastly, we calculate  $\int_0^\infty$. Recalling \eqref{polyP},
\begin{equation}\label{ppo}
  \int_{0}^{\infty} e^{-w} P_s(w) w^{(s+1)/\mu-1}\, dw  = \sum_{\ell = 0}^s c_{s,\ell} \G\left( \frac{s+1}{\mu} + \ell \right)
\end{equation}
where $c_{s,\ell}$ is the coefficient of $(z-z_0)^s$ in the Taylor expansion of $q(z)\phi(z)^\ell/\ell!$. Therefore \eqref{ppo} is the coefficient of $(z-z_0)^s$ in
\begin{align}
  q(z)\sum_{\ell = 0}^s  \G\left( \frac{s+1}{\mu} + \ell \right)\frac{\phi(z)^\ell}{\ell!}  & =
  q(z) \G\left( \frac{s+1}{\mu} \right)\sum_{\ell = 0}^s  \frac{\G\left( \frac{s+1}{\mu} + \ell \right)}{\G\left( \frac{s+1}{\mu} \right)\ell!}(-1)^\ell  (-\phi(z))^\ell\notag\\
   & = q(z) \G\left( \frac{s+1}{\mu} \right)\sum_{\ell = 0}^s  \binom{-(s+1)/\mu}{\ell}(-\phi(z))^\ell. \label{ppoo}
\end{align}
Extending this sum to infinity will not affect the coefficient of $(z-z_0)^s$ and so we may replace \eqref{ppoo} by
\begin{equation*}
  q(z) \G\left( \frac{s+1}{\mu} \right)\sum_{\ell = 0}^\infty  \binom{-(s+1)/\mu}{\ell}(-\phi(z))^\ell
  = q(z) \G\left( \frac{s+1}{\mu} \right) (1-\phi(z))^{-(s+1)/\mu}.
\end{equation*}
This completes the proof of Proposition \ref{min4}.
\end{proof}
Our main Theorem \ref{il} now follows from \eqref{thre2}, Lemma \ref{min3} and Proposition \ref{min4}.
\end{proof}

\section{Second proof of Theorem \ref{il}} \label{spf}

This proof of Theorem \ref{il} is based on Olver's \cite[Thm. I]{Ol70} or equivalently \cite[Thm. 6.1, p. 125]{Ol}. Instead of employing the substitution $w=N p_0 (z-z_0)^\mu$, Olver uses $v=p(z_0)-p(z)$ as in the usual proofs of Laplace's method (see Section \ref{laplace}).

To get the result to match the statement of Theorem \ref{il}, we have to treat the branch factor $e^{2\pi i k/\mu}$ more explicitly than in \cite[Thm. 6.1, p. 125]{Ol}. The coefficients $\alpha_s$ naturally appear in a power series in this proof and we use a method inspired by the application of Cauchy's differentiation formula in \cite{CFW} to obtain Perron's expression for them.


\SpecialCoor
\psset{griddots=5,subgriddiv=0,gridlabels=0pt}
\psset{xunit=1cm, yunit=1cm, runit=1cm}
\psset{linewidth=1pt}
\psset{dotsize=5pt 0,dotstyle=*}

\begin{figure}[ht]
\begin{center}
\begin{pspicture}(-5,-2)(10,3) 

\psset{arrowscale=2,arrowinset=0.5}

\psarc[linecolor=blue](0,0){2.7}{-47.7946}{227.795}
\pscircle[linecolor=gray](0,0){2}
\pscircle[fillstyle=solid,fillcolor=light](0,0){1}

\psline[linecolor=gray,linestyle=dashed]{->}(0,0)(4,2.3094)
\psline[linestyle=dotted](0,1)(7,1)
\psline[linestyle=dotted](6.7868, -1.98477)(-0.142132, -0.989848)
\psline{->}(0,0)(0.866, 0.5)
\pscircle[fillstyle=solid,fillcolor=light](7,-0.5){1.5}

\psdots(2.33827, 1.35)(0.866, 0.5)(0,0)(7,-0.5)(0.4,-0.5)(7.4,-1.3)
\rput(4.4,2.3){$\theta_k$}
\rput(-0.9,1.4){$\nd$}
\rput(2.7,1.25){$b$}
\rput(1.166,0.35){$b'$}
\rput(0.1,-0.65){$z$}
\rput(-0.3,-0.15){$z_0$}
\rput(-3.4,-1){radius $R_p$}
\rput(0,-1.7){radius $\rho$}
\rput(6.7,-0.65){$0$}
\rput(7.1,-1.45){$\tau$}
\rput(4.2,0.5){$z=z_0+g(\tau)$}
\rput(4.2,-0.8){$\tau=\tau(z)$}
\rput(1.2,-0.6){$\nd_z$}
\rput(8.7,-1.3){$\nd_\tau$}

\psline{->}(4.7,0)(3.7, 0)
\psline{<-}(4.7,-1.3)(3.7, -1.3)

\end{pspicture}
\caption{The regions $\nd_z$ and $\nd_\tau$}
\label{pf2fig}
\end{center}
\end{figure}

\begin{proof}[Second proof of Theorem \ref{il}]
Let
\begin{equation*}
    \nd := \{z\in \C \ : \ |z-z_0| \lqs \rho\}
\end{equation*}
for  $\rho =R_p$, initially. As in \eqref{kphi}, \eqref{krho} we may decrease $\rho$ to ensure that
\begin{equation}\label{kphix}
    |\phi(z)| \lqs 1/2 \quad \text{for all} \quad z \in \nd.
\end{equation}

By Proposition \ref{ilz} we only need to estimate the integral
\begin{equation} \label{only}
    \int_{z_0}^{b} e^{N (p(z)-p(z_0))}  q(z) \, dz
\end{equation}
where $b$ is on the bisecting line with angle $\theta_k$ and a distance $R_p$ from $z_0$.
We will use  the change of variables $v:=p(z_0)-p(z)$ and,
to prepare for this, set
\begin{equation} \label{taudef}
    \tau = \tau(z) := p_0^{1/\mu}(z-z_0)(1-\phi(z))^{1/\mu}
\end{equation}
with all roots  principal. By \eqref{f} it is clear that $\tau$ is some $\mu$th root of $p(z_0)-p(z)$. We also see by \eqref{kphix} that $\tau$ is a holomorphic function of $z$ for $z$ in  $\mathcal D$. We have $\left.\frac{d\tau}{dz}\right|_{z=z_0} = p_0^{1/\mu} \neq 0$ and consequently, by the Inverse Function Theorem for holomorphic functions, there exists a neighborhood $\mathcal D_\tau$ of $0$ so that $z$ is a holomorphic function of $\tau$ there:
\begin{equation*}
    z-z_0=g(\tau):=\sum_{s=1}^\infty c_s \tau^s \qquad (\tau \in \mathcal D_\tau).
\end{equation*}
Choose $\mathcal D_\tau$ to be a disk centered at $0$ and small enough that the image $\mathcal D_z:=z_0+g(\mathcal D_\tau)$ is contained in $\mathcal D$. See Figure \ref{pf2fig}, ($\mathcal D_z$ may not be a disk).
Since
\begin{equation*}
     p'(z)=-\sum_{s=0}^\infty (s+\mu) p_s (z-z_0)^{s+\mu-1}
\end{equation*}
 we have
\begin{equation} \label{mur}
    -\frac{q(z)}{p'(z)}=-\frac{q(z_0+g(\tau))}{p'(z_0+g(\tau))} =: F(\tau) = \sum_{s=0}^\infty \beta_s \tau^{s -\mu+1} \qquad (z\in \mathcal D_z, \tau \in \mathcal D_\tau).
\end{equation}
Shrink $\mathcal D_\tau$ (and correspondingly $\mathcal D_z$) if necessary so that  $\tau^{\mu-1} F(\tau)$ is holomorphic on $\mathcal D_\tau$; we are avoiding any zeros of $p'(z)$ away from $z=z_0$. Taylor's theorem implies there exist constants $K_{F,S}$ such that
\begin{equation}\label{kfs}
    \left| \tau^{\mu-1} F(\tau) - \sum_{s=0}^{S-1} \beta_s \tau^s\right| \lqs K_{F,S} |\tau|^S \qquad (\tau \in \mathcal D_\tau, S \in \Z_{\gqs 0}).
\end{equation}
To understand the dependence of $K_{F,S}$ on $q$ we may write the remainder term explicitly as
\begin{equation*}
    \tau^{\mu-1} F(\tau) = \sum_{s=0}^{S-1} \beta_s \tau^s +\frac{\tau^S}{2\pi i} \int_{C_0} \frac{w^{\mu-1} F(w)}{w^S(w-\tau)}\, dw
\end{equation*}
with $C_0$ the boundary of $\mathcal D_\tau$, oriented positively. Since
\begin{equation}\label{tenn}
    \frac{1}{2\pi i} \int_{C_0} \frac{w^{\mu-1} F(w)}{w^S(w-\tau)}\, dw = -\frac{1}{2\pi i} \int_{C_0} \frac{q(z_0+g(w))}{p'(z_0+g(w)) \cdot w^{S-\mu+1}(w-\tau)}\, dw
\end{equation}
and $|q(z_0+g(w))|\lqs K_q$ on the right of \eqref{tenn}, we may write $K_{F,S}=K^*_{F,S}\cdot K_q$ with $K^*_{F,S}$ independent of $q$. For these estimates we have shrunk $\mathcal D_\tau$ (and  $\mathcal D_z$) again, for example to half their size, so that $w-\tau$ in \eqref{tenn} is bounded away from zero for $w\in C_0$ and $\tau \in \mathcal D_\tau$.

\begin{lemma} \label{fca}
For all $z \in  \mathcal D_z$ with $z$ also on the line between $z_0$ and $b$, we have
\begin{equation}\label{tau}
    \tau(z) = e^{2\pi i k/\mu} \bigl(p(z_0)-p(z)\bigr)^{1/\mu}.
\end{equation}
\end{lemma}
\begin{proof}
Recall that $\arg(p_0)=\omega_0$ and $\arg(z-z_0)=\theta_k$. Hence  $\arg(p_0(z-z_0)^\mu)=0$ and so
\begin{align*}
  \tau(z)  & := p_0^{1/\mu}(z-z_0)\bigl(1-\phi(z)\bigr)^{1/\mu}  \\
   & =   p_0^{1/\mu}(z-z_0)\bigl(p_0(z-z_0)^\mu\bigr)^{-1/\mu} \bigl(p(z_0)-p(z)\bigr)^{1/\mu}\\
   & =  e^{i \omega_0/\mu}\cdot e^{i \theta_k} \left| p_0^{1/\mu}(z-z_0)\bigl(p_0(z-z_0)^\mu \bigr)^{-1/\mu} \right| \bigl(p(z_0)-p(z)\bigr)^{1/\mu}\\
   & = e^{2\pi i k/\mu} \bigl(p(z_0)-p(z)\bigr)^{1/\mu}
\end{align*}
as desired.
\end{proof}

Fix $b'$ on the line between $z_0$ and $b$ so that the segment from $z_0$ to $b'$ is contained in $\mathcal D_z$.  Hence Lemma \ref{fca} shows that we have
\begin{equation*}
    v:=p(z_0)-p(z), \quad \tau=e^{2\pi i k/\mu} v^{1/\mu}, \quad v=\tau^\mu
\end{equation*}
for all $z$ and $\tau(z)$ where $z$ is on the line between $z_0$ and $b'$.

To estimate \eqref{only} we see first that $\int_{b'}^b$ is $O(K_q e^{-\varepsilon' N})$ as in \eqref{cby}.
Using $\frac{dv}{dz}=-p'(z)$ and \eqref{mur} we find
\begin{align}
    \int_{z_0}^{b'}   e^{N (p(z)-p(z_0))}  q(z) \, dz & = \int_0^{p(z_0)-p(b')}   e^{-N v}  q(z) \cdot \frac{dz}{dv} \, dv \notag\\
    & = -\int_0^{p(z_0)-p(b')}   e^{-N v}   \frac{q(z)}{p'(z)} \, dv\notag\\
    & = \int_0^{p(z_0)-p(b')}   e^{-N v}    F(e^{2\pi i k/\mu} v^{1/\mu}) \, dv. \label{use}
\end{align}
The contour of integration in \eqref{use} is the image of the line between $z_0$ and $b'$ in the $v$-plane. Except for the starting point, this contour is contained in the half-plane with positive real part by \eqref{c1}. The principal root $v^{1/\mu}$ is holomorphic in this half-plane and therefore the integrand in \eqref{use} is holomorphic there too.  Set $w:=p(z_0)-p(b')$. By Cauchy's Theorem we may change the contour of integration to  the straight line from $0$ to $w$. (The integrand may have a singularity at $v=0$, but it is $\ll |v|^{1/\mu -1}$  for $|v|$ small, and so moving the path of integration near $0$ may be justified.) Employing \eqref{kfs} yields
\begin{equation} \label{berhi}
    \int_0^{w}   e^{-N v}    F(e^{2\pi i k/\mu} v^{1/\mu}) \, dv =  \sum_{s=0}^{S-1} \beta_s e^{2\pi i k(s+1)/\mu}
    \int_0^{w}   e^{-N v}     v^{(s+1)/\mu-1} \, dv +E_S
\end{equation}
with
\begin{align}
|E_S| & \lqs K_{F,S} \int_0^{w}   \left| e^{-N v}\right|      |v|^{(S+1)/\mu-1} \, |dv| \notag\\
     & = K_{F,S} \int_0^{|w|}   e^{-N \Re(w/|w|) t}     t^{(S+1)/\mu-1} \, dt  \notag\\
     & \lqs K_{F,S} \G\left(\frac{S+1}{\mu}\right) \left(N \Re\left(\frac{w}{|w|}\right)\right)^{-(S+1)/\mu} \label{port}
\end{align}
on extending the limit of integration to infinity. The next lemma estimates the integral in \eqref{berhi}.

\begin{lemma} \label{jna}
Suppose $N,$ $r,$ $\varepsilon >0$ and $\Re(w)\gqs \varepsilon$. For an implied constant depending only on $r$ and $w$ we have
\begin{equation*}
    \int_0^{w}   e^{-N v}     v^{r-1} \, dv = N^{-r}\left(\G(r)+O\left( e^{-\varepsilon N/2}\right) \right).
\end{equation*}
\end{lemma}
\begin{proof}
Continue the line of integration to $w\infty$ and write $\int_0^{w} = \int_0^{w\infty}-\int_{w}^{w\infty}$. The integral $\int_0^{w\infty}$ is computed by rotating the line of integration to $\R_{\gqs 0}$ which is straightforward to justify:
\begin{equation*}
    \int_0^{w\infty}   e^{-N v}     v^{r-1} \, dv = N^{-r}\G(r).
\end{equation*}
The absolute value of  $\int_{w}^{w\infty}$ is bounded by
\begin{align*}
    |w|^r \left| \int_1^\infty e^{-Nwt}t^{r-1}\, dt \right| & \lqs |w|^r  \int_1^\infty e^{-N\varepsilon t}t^{r-1}\, dt\\
    & = \frac{|w|^r}{(N\varepsilon)^r}  \int_{N\varepsilon}^\infty e^{-u}u^{r-1}\, du\\
     & \lqs \left(\frac{2|w|}{\varepsilon} \right)^{r} \G(r) N^{-r} e^{-\varepsilon N/2}
\end{align*}
where the last line used \eqref{gam}.
\end{proof}
We have shown so far, with \eqref{use}, \eqref{berhi}, \eqref{port} and with Lemma \ref{jna} applied to \eqref{berhi}, that
\begin{equation*}
    \int_{z_0}^{b'}   e^{N (p(z)-p(z_0))}  q(z) \, dz =  \sum_{s=0}^{S-1} \beta_s e^{2\pi i k(s+1)/\mu} \frac{\G((s+1)/\mu)}{N^{(s+1)/\mu}} +E^*_S
\end{equation*}
where
\begin{equation} \label{estr}
    E^*_S \ll  \frac{K_{F,S}}{N^{(S+1)/\mu}} +\left(\sum_{s=0}^{S-1} |\beta_s|\right) e^{-\varepsilon N/2}
\end{equation}
for an implied constant independent of $N$ and $q$. A similar argument to the one after \eqref{tenn}, showing that $K_{F,S}/K_q$ may be bounded independently of $q$,  shows that $|\beta_s|/K_q$ is also independent of $q$ since
\begin{align}
    \beta_s & = \frac{1}{2\pi i} \int_{C_0} \frac{\tau^{\mu-1} F(\tau)}{\tau^{s+1}} \, d\tau \notag\\
     & = -\frac{1}{2\pi i} \int_{C_0} \frac{q(z_0+g(\tau))}{p'(z_0+g(\tau)) \cdot \tau^{s-\mu+2}} \, d\tau. \label{dgf}
\end{align}
We have already seen that integral $\int_{b'}^b$ has exponential decay in $N$,  and so may be included in the error term \eqref{estr}.
Consequently
\begin{equation}\label{icax}
     \int_{\cc} e^{N \cdot p(z)} q(z) \, dz = e^{N \cdot p(z_0)} \left(\sum_{s=0}^{S-1} \G\left(\frac{s+1}{\mu}\right)  \frac{\beta_s \cdot e^{2\pi i k(s+1)/\mu}}{N^{(s+1)/\mu}}+
     O\left( \frac{K_q}{N^{(S+1)/\mu}}\right)
     \right)
\end{equation}
as desired.

It only remains to compute the numbers $\beta_s$. A change of variables in \eqref{dgf} shows
\begin{equation*}
    \beta_s = -\frac{1}{2\pi i} \int_{C_{z_0}} \frac{q(z)}{p'(z) \cdot \tau^{s-\mu+2}} \frac{d \tau}{d z}\, dz
\end{equation*}
for  $C_{z_0} \subset \mathcal D_z$  a positively oriented circle centered at $z_0$.
Use \eqref{f} and \eqref{taudef} to show that
\begin{equation*}
    \frac{d \tau}{d z} = -\frac{1}{\mu} \tau^{1-\mu}p'(z).
\end{equation*}
Hence
\begin{align}
    \beta_s & = \frac{1}{2\pi i \cdot \mu} \int_{C_{z_0}} \frac{q(z)}{\tau^{s+1}} \, dz \notag\\
    & = \frac{1}{2\pi i \cdot \mu }  p_0^{-(s+1)/\mu} \int_{C_{z_0}} \frac{q(z)\cdot (1-\phi(z))^{-(s+1)/\mu}}{(z-z_0)^{s+1}} \, dz \label{junn}\\
    & = \frac{1}{\mu  \cdot s!}  p_0^{-(s+1)/\mu} \frac{d^s}{dz^s}\left\{q(z) \left( 1-\phi(z)\right)^{-(s+1)/\mu)} \right\}_{z=z_0}
    \label{jun}
\end{align}
where \eqref{junn} is related to \eqref{jun}  by Cauchy's differentiation formula. Thus $\beta_s$ is recognized as $\alpha_s$ from \eqref{as}.
Combining \eqref{icax} and \eqref{jun} completes the second proof of Theorem \ref{il}.
\end{proof}

\section{An important case}
A case of Corollary \ref{il2} that often arises is when $\cc$ passes through the saddle-point $z_0$ in a straight line or in a curve with a well-defined tangent at $z_0$. If $\mu$ is even then these paths will pass through opposite valley sectors, for example with $\theta_{k_2}=\theta_{k_1}\pm \pi$. In this case the terms in \eqref{wimab} with $s$ odd vanish:

\begin{cor} \label{ilf} {\rm (Perron's method for a holomorphic integrand with contour passing  through  a maximum between opposite sectors.)}
Suppose Assumptions \ref{ma0} hold and $\mu$ is even. Let $\cc$ be a contour beginning at $z_1$, passing through $z_0$ and ending at $z_2$, with these points all distinct.  Suppose that
\begin{equation*}
   \Re(p(z))<\Re(p(z_0)) \quad \text{for all} \quad z \in \cc, \ z\neq z_0.
\end{equation*}
Let $\cc$ approach $z_0$ in a sector of angular width $2\pi/\mu$ about $z_0$ with bisecting angle $\theta_{k}+(2n+1)\pi$  for some $n\in \Z$, and  initially leave $z_0$ in a sector of the same size with bisecting angle $\theta_{k}$.
 Then for every $M \in \Z_{\gqs 0}$,
\begin{equation*} 
    \int_\cc e^{N \cdot p(z)} q(z) \, dz = e^{N \cdot p(z_0)} \left(\sum_{m=0}^{M-1}  \G\left(\frac{2m+1}{\mu}\right) \frac{2\alpha_{2m}   \cdot e^{2\pi i k (2m+1)/\mu}}{N^{(2m+1)/\mu}} + O\left(\frac{K_q}{N^{(2M+1)/\mu}} \right)  \right)
\end{equation*}
as $N \to \infty$ where the implied constant  is independent of $N$ and $q$. The numbers $\alpha_{s}$ are given by \eqref{as}.
\end{cor}
\begin{proof}
Apply Corollary \ref{il2} with $k_2=k$ and $k_1=k+(2n+1)\mu/2$. Then the difference of exponentials in \eqref{wimab} is
\begin{align*}
    e^{2\pi i k (s+1)/\mu} - e^{2\pi i (k+(2n+1)\mu/2) (s+1)/\mu} & = e^{2\pi i k (s+1)/\mu} \left(1-e^{2\pi i (\mu/2)(s+1)/\mu} \right) \\
    & = e^{2\pi i k (s+1)/\mu} \left(1-(-1)^{s+1} \right)
\end{align*}
and the corollary  follows on writing $s=2m$.
\end{proof}

The above result corresponds to \cite[Thm. 7.1, p. 127]{Ol} when $\mu=2$,   giving a clearer description of how the result depends on $\cc$ near $z_0$. Olver does not give the formula \eqref{as} for the coefficients and perhaps he was not aware of Perron's paper \cite{Pe17}. It does not appear in the references of \cite{Ol}, though \cite{Bu14} is listed. Perron's paper \cite{Pe17} is not cited by the classic works \cite{deB,Cop,Din} either. It is briefly mentioned, along with \cite{Bu14}, in section 2.4 of Erd\'elyi's book \cite{Erd}, though in a way which seems to imply that Perron only gives the main term of the asymptotic expansions.

\section{Generalizations} \label{general}

\subsection{Including a factor $(z-z_0)^{a-1}$ with $\Re(a)>0$}
Perron's  results in \cite{Pe17} cover a more general situation where we have $(z-z_0)^{a-1} q(z)$ in the integrand,  instead of just $q(z)$. Unlike Perron, we do not assume that $q(z_0) \neq 0$. The number $a$ is in $\C$ and so we must pay attention to which branch of $(z-z_0)^{a-1}$ is meant. For example, if $z$ is on the  bisecting line with angle $\theta_k$ (recall \eqref{bisec}) then possible branches are
\begin{equation} \label{thh}
    (z-z_0)^{a-1} = |z-z_0|^{a-1}\cdot e^{i\theta_\ell(a-1)}
\end{equation}
for $\ell \in \Z$ with $\ell \equiv k \bmod \mu$. The principal value of the power \eqref{thh} has the unique such $\ell$ for which $\theta_\ell$ is in $(-\pi,\pi]$.

The standard method for integrating  a multi-valued function such as \eqref{thh} along a contour $\cc$  is to begin with a  specified branch, and as $z$ moves along $\cc$ the branch is determined by continuity. In particular, if $z-z_0$ crosses the negative real axis then $(z-z_0)^{a-1}$ enters another branch.

\begin{theorem}  \label{m4} {\rm (Perron's method for an integrand containing a factor $(z-z_0)^{a-1}$ for $\Re(a)>0$ and with contour starting at a maximum.)}
Suppose Assumptions \ref{ma0} hold.
 Let $\cc$ be a contour from $z_0$ to $z_1$, with $z_0 \neq z_1$, that initially runs along the bisecting line with angle $\theta_k$ for some $k \in \Z$. Suppose  $\Re(a)>0$ and
    that
\begin{equation}\label{c1m4}
     \Re(p(z))<\Re(p(z_0)) \quad \text{for all} \quad z \in \cc, \ z\neq z_0.
\end{equation}
On the initial part of $\cc$ we take
\begin{equation} \label{thhxm4}
    (z-z_0)^{a-1} = |z-z_0|^{a-1}\cdot e^{i\theta_k(a-1)}.
\end{equation}
 Then for any $S \in \Z_{\gqs 0}$,
\begin{equation} \label{wimm4}
    \int_\cc e^{N \cdot p(z)} (z-z_0)^{a-1} q(z) \, dz = e^{N \cdot p(z_0)} \left(\sum_{s=0}^{S-1} \G\left(\frac{s+a}{\mu}\right) \frac{\alpha_s \cdot e^{2\pi i k (s+a)/\mu}}{N^{(s+a)/\mu}} + O\left(\frac{K_q}{N^{(S+\Re(a))/\mu}} \right)  \right)
\end{equation}
where the implied constant in \eqref{wimm4} is independent of $N$ and $q$. The numbers $\alpha_s$
  are given by
\begin{equation} \label{as2}
    \alpha_s = \frac{1}{\mu \cdot  s!} p_0^{-(s+a)/\mu} \frac{d^s}{dz^s}\left\{q(z) \cdot \left( 1-\phi(z)\right)^{-(s+a)/\mu} \right\}_{z=z_0}.
\end{equation}
\end{theorem}

The condition in Theorem \ref{m4} that $\cc$  initially runs along the bisecting line with angle $\theta_k$ is not really necessary and just included for convenience. The theorem is true if $\cc$ begins in the sector of angular width $2\pi/\mu$ about $z_0$ with this bisecting line, and the branch of $(z-z_0)^{a-1}$ is consistent with \eqref{thhxm4}. The $a=1$ case of Theorem \ref{m4}  is Theorem \ref{il} and, in particular, \eqref{as2} reduces to \eqref{as} when $a=1$.

\begin{proof}[Proof of Theorem \ref{m4}]
We may use a straightforward extension of the first proof of Theorem \ref{il} given in Section \ref{fp}. The key step is in Lemma \ref{min3}, where we need to express $(z-z_0)^{a+n}$ in terms of $w$ for any $n\in \Z$ and $z$ on the bisecting line with angle $\theta_{k}$. Here, $(z-z_0)^{a+n} = (z-z_0)^{a-1}\cdot (z-z_0)^{1+n}$ where $(z-z_0)^{a-1}$ is given by \eqref{thhxm4} and $(z-z_0)^{1+n}$ is unambiguous. Then
\begin{align}
  (z-z_0)^{a+n} &= |z-z_0|^{a+n}\cdot e^{i\theta_k(a+n)} \notag\\
   &= \left(\frac{w}{N|p_0|}\right)^{(a+n)/\mu} e^{-i \omega_0 (a+n)/\mu} e^{2\pi i k(a+n)/\mu} \notag\\
   &= w^{(a+n)/\mu} (N p_0)^{-(a+n)/\mu}  e^{2\pi i k(a+n)/\mu} \label{xper}
\end{align}
with the powers in \eqref{xper} taking the principal values.
Therefore
\begin{align*}
  I_s(N) & := \int_{z_0}^{b'} e^{-w} P_s(w) (z-z_0)^{s+a-1} \, dz \\
   & = \frac{e^{2\pi i k (s+a)/\mu}}{\mu \cdot(N p_0)^{(s+a)/\mu}} \int_0^{N|p_0|(\rho/2)^\mu} e^{-w} P_s(w) w^{(s+a)/\mu-1}\, dw.
\end{align*}
The rest of the proof continues as in Section \ref{fp} to obtain the result.
\end{proof}

The second proof  given in Section \ref{spf} may also be adapted to Theorem \ref{m4}. The series $F(\tau)$ has a more complicated construction as described next. Define $\tau$ as in \eqref{taudef} and choose the branch of $\tau^{a-1}$ so that
\begin{equation} \label{taudef2}
    \tau^{a-1} = \tau(z)^{a-1} := p_0^{(a-1)/\mu}(z-z_0)^{a-1}(1-\phi(z))^{(a-1)/\mu}
\end{equation}
where $(z-z_0)^{a-1}$ is consistent with \eqref{thhxm4} and the two other powers in \eqref{taudef2} are principal. Then
\begin{equation*}
  -\frac{(z-z_0)^{a-1}}{\tau^{a-1}} q(z) \frac{\tau^{\mu-1}}{p'(z)} = - \frac{ q(z)}{p_0^{(a-1)/\mu}(1-\phi(z))^{(a-1)/\mu}} \frac{\tau^{\mu-1}}{p'(z)} = h(z)
\end{equation*}
for $h(z)$ holomorphic on $\mathcal{D}_z$. As in \eqref{mur} we may write
\begin{equation*}
  h(z_0+g(\tau)) =: \tau^{\mu -1} F(\tau) = \sum_{s=0}^\infty \beta_s \tau^s \qquad (z\in \mathcal D_z, \tau \in \mathcal D_\tau),
\end{equation*}
implying the identity
\begin{equation*}
  -(z-z_0)^{a-1}\frac{q(z)}{p'(z)} = \tau^{a -1} F(\tau) \qquad (z\in \mathcal D_z, \tau \in \mathcal D_\tau).
\end{equation*}
A calculation similar to Lemma \ref{fca} shows that
\begin{equation*}
   \tau^{a-1} = e^{2\pi i k(a-1)/\mu} v^{(a-1)/\mu}
\end{equation*}
when $z$ in $\mathcal D_z$ is on the line from $z_0$ to $b$.

With the above results in place, the rest of the proof of Section \ref{spf} goes through easily. Of particular interest is the computation of $\beta_s$, as in the equations leading to \eqref{dgf} and \eqref{jun}:
\begin{align}
    \beta_s & = \frac{1}{2\pi i} \int_{C_0} \frac{\tau^{\mu-1} F(\tau)}{\tau^{s+1}} \, d\tau \notag\\
     & = -\frac{1}{2\pi i} \int_{C_0}
     \frac{q(z_0+g(\tau))}{p_0^{(a-1)/\mu}(1-\phi(z_0+g(\tau)))^{(a-1)/\mu}  \cdot p'(z_0+g(\tau)) \cdot \tau^{s-\mu+2}} \, d\tau \label{dgf2}\\
     & = -\frac{1}{2\pi i \cdot \mu \cdot p_0^{(a-1)/\mu}} \int_{C_{z_0}}
     \frac{q(z)}{(1-\phi(z))^{(a-1)/\mu}  \cdot \tau^{s+1}} \, dz \notag\\
      & = -\frac{1}{2\pi i \cdot \mu \cdot p_0^{(s+a)/\mu}} \int_{C_{z_0}}
     \frac{q(z)(1-\phi(z))^{-(s+a)/\mu}}{ (z-z_0)^{s+1}} \, dz = \alpha_s. \notag
\end{align}
Formula \eqref{dgf2} will be used in Proposition \ref{albnd}. When $a=1$ then \eqref{dgf2} reduces to \eqref{dgf}.

\subsection{Including a factor $(z-z_0)^{a-1}$ with arbitrary $a \in \C$}
Two applications of Theorem \ref{m4} give the following corollary.

\begin{cor}  \label{m5} {\rm (Perron's method for an integrand containing a factor $(z-z_0)^{a-1}$ for $\Re(a)>0$ and with contour  passing  through  a maximum.)}
Suppose Assumptions \ref{ma0} hold.
Let $\cc$ be a contour starting at $z_1$, passing through $z_0$ and ending at  $z_2$,  with these three points all distinct. Suppose there are $k_1,$ $k_2 \in \Z$ so that, in a neighborhood of $z_0$,  $\cc$ runs along  the bisecting line with angle $\theta_{k_1}$ as $\cc$ approaches $z_0$ and $\cc$ runs along  the bisecting line with angle $\theta_{k_2}$  leaving $z_0$.  Assume  $\Re(a)>0$ and  that
\begin{equation}\label{c1m5}
     \Re(p(z))<\Re(p(z_0)) \quad \text{for all} \quad z \in \cc, \ z\neq z_0.
\end{equation}
On the part of  $\cc$ approaching $z_0$ we take
\begin{equation} \label{thhxm5a}
    (z-z_0)^{a-1} = |z-z_0|^{a-1}\cdot e^{i\theta_{k_1}(a-1)}
\end{equation}
and  on the part of  $\cc$ leaving $z_0$,
\begin{equation} \label{thhxm5b}
    (z-z_0)^{a-1} = |z-z_0|^{a-1}\cdot e^{i\theta_{k_2}(a-1)}.
\end{equation}
Then for any $S \in \Z_{\gqs 0}$,
\begin{multline} \label{wimm5}
    \int_\cc e^{N \cdot p(z)} (z-z_0)^{a-1} q(z) \, dz \\
    = e^{N \cdot p(z_0)} \left(\sum_{s=0}^{S-1} \G\left(\frac{s+a}{\mu}\right) \frac{\alpha_s  \left( e^{2\pi i {k_2}(s+a)/\mu}- e^{2\pi i {k_1}(s+a)/\mu}\right)}{N^{(s+a)/\mu}} + O\left(\frac{K_q}{N^{(S+\Re(a))/\mu}} \right)  \right)
\end{multline}
where the implied constant in \eqref{wimm5} is independent of $N$ and $q$. The numbers $\alpha_s$
 are given by \eqref{as2}.
\end{cor}

The next result is an elegant extension of Corollary \ref{m5}, where Perron shows that the condition $\Re(a)>0$ may be dropped provided that  the contour of integration is adjusted to make sure it avoids $z_0$. We will need  this extension for the examples in Sections \ref{egg} and \ref{egg2}.

\begin{theorem}  \label{m6} {\rm (Perron's method for an integrand containing a factor $(z-z_0)^{a-1}$  for arbitrary $a \in \C$.)}
Suppose Assumptions \ref{ma0} hold.
Let $\cc$ be the following contour. Starting at $z_1$ it runs to the point $z'_1$ which is a distance $R_p$ from $z_0$ and  on the bisecting line with angle $\theta_{k_1}$. Then the contour circles $z_0$ to arrive at the point $z'_2$ which is a distance $R_p$ from $z_0$ and  on the bisecting line with angle $\theta_{k_2}$. Finally, the contour ends at $z_2$. The integers $k_1$ and $k_2$ keep track of how $\cc$ rotates about $z_0$ between $z'_1$ and $z'_2$; the angle of rotation is  $2\pi(k_2-k_1)/\mu$.

Suppose  that
$
    \Re(p(z))<\Re(p(z_0))
$
for all $z$ in the segments of $\cc$ between $z_1$ and $z'_1$ and between $z'_2$ and $z_2$ (including endpoints).
Let $a \in \C$. For $z\in \cc$, the branch of $(z-z_0)^{a-1}$ is specified by requiring
\begin{equation} \label{thhxm6a}
    (z'_1-z_0)^{a-1} = |z'_1-z_0|^{a-1}\cdot e^{i\theta_{k_1}(a-1)}
\end{equation}
when $z=z'_1$ and by continuity at the other points of $\cc$.
Then for any $S \in \Z_{\gqs 0}$, \eqref{wimm5} holds
with an implied constant  independent of $N$ and $q$.
 If $(s+a)/\mu \in \Z_{\lqs 0}$ then
 \begin{equation*}
    \G((s+a)/\mu) \left( e^{2\pi i {k_2}(s+a)/\mu}- e^{2\pi i {k_1}(s+a)/\mu}\right)
 \end{equation*}
  in \eqref{wimm5} is not defined and  must be replaced by $2\pi i (k_2-k_1)(-1)^{(s+a)/\mu}/|(s+a)/\mu|!$.
  \end{theorem}
\begin{proof}
We will follow \cite[Sect. 4]{Pe17} and the first proof of Theorem \ref{il} given in Section \ref{fp}. It is convenient to move $z'_1$, $z'_2$ and the circular path of integration to the smaller radius $\rho/2$ with $\rho$ satisfying \eqref{krho}. The points $z'_1$ and $z'_2$ are kept on their bisecting lines.

There exists $\varepsilon>0$ so that $\Re(p(z)-p(z_0)) \lqs -\varepsilon$ for  all $z$ in the segment of $\cc$ between $z_1$ and $z'_1$ (using \eqref{repl>} for the new part). It also follows that on this segment $z$ is bounded away from $z_0$. Hence
\begin{equation*}
  \left| \int_{z_1}^{z'_1} e^{N(p(z)-p(z_0))} (z-z_0)^{a-1} q(z) \, dz \right|
  \lqs K_q e^{-\varepsilon N} \int_{z_1}^{z'_1} \left|(z-z_0)^{a-1}\right|  \, |dz| \ll K_q e^{-\varepsilon N}.
\end{equation*}
We obtain a similar bound for the integral between $z'_2$ and $z_2$. The integral around the circular path from $z'_1$ to $z'_2$ remains to be estimated.

Following Lemma \ref{pq}, write the integrand in the form
\begin{equation} \label{gftt}
    e^{N (p(z)-p(z_0))}(z-z_0)^{a-1} q(z) = \sum_{s=0}^{S-1} e^{-w} P_s(w) (z-z_0)^{s+a-1} + (z-z_0)^{a-1} Q_S(z)
\end{equation}
with  $w=N p_0 (z-z_0)^\mu$ as in \eqref{ww}. The integer $S$ should satisfy $S \gqs 0$ and  $S+\Re(a)>0$.

\begin{lemma} \label{mintt}
With this choice of $S$,
\begin{equation*}
    \int_{z'_1}^{z'_2}(z-z_0)^{a-1} Q_S(z) \, dz =O\left(\frac{K_q}{N^{(S+\Re(a))/\mu}}\right).
\end{equation*}
\end{lemma}
\begin{proof}
We may change the path of integration, moving the circular part closer to $z_0$ as follows. From $z'_1$ the new path follows the bisecting line with angle $\theta_{k_1}$ to a point $\zeta_1$ close to $z_0$. Then it circles $z_0$ until reaching $\zeta_2$ on the bisecting line with angle $\theta_{k_2}$. This bisecting line is followed to $z'_2$.

As in Lemma \ref{pq},
\begin{equation*}
    Q_S(z)=\frac{(z-z_0)^S}{2\pi i}\int_\g \frac{q(\tau)e^{-w+w \phi(\tau)}}{(\tau-z_0)^S(\tau-z)}\, d\tau
\end{equation*}
where $\g$ is the positively oriented circle of radius $\rho$ about $z_0$.
Note that
\begin{equation*}
  \Re(-w+w \phi(\tau)) \lqs |w|(1+|\phi(\tau)|) \lqs 3|w|/2 \lqs 2N|p_0||z-z_0|^\mu.
\end{equation*}
Hence, for $z$ with $|z-z_0|\lqs \rho/2$,
\begin{equation*}
  (z-z_0)^{a-1} Q_S(z)  \ll K_q |z-z_0|^{S+\Re(a)-1} e^{2N|p_0||z-z_0|^\mu}.
\end{equation*}
Suppose $\zeta_1$, $\zeta_2$ and the circular path of integration between them are at a distance $r$ from $z_0$. Then
\begin{equation} \label{jnn}
    \int_{\zeta_1}^{\zeta_2}(z-z_0)^{a-1} Q_S(z) \, dz =O\left(K_q \cdot r^{S+\Re(a)} e^{2N|p_0|r^\mu}\right).
\end{equation}
Choosing any $r \lqs N^{-1/\mu}$ shows that \eqref{jnn} satisfies the lemma's bound.
 The remaining integrals along the bisecting lines may now be bounded using \eqref{gf3} as in Lemma \ref{min}, completing the proof.
\end{proof}

Our work so far has shown
\begin{equation} \label{kqa}
    \int_\cc e^{N \cdot p(z)} (z-z_0)^{a-1} q(z) \, dz  = e^{N \cdot p(z_0)}\left( \sum_{s=0}^{S-1} I^*_s(N) +  O\left(\frac{K_q}{N^{(S+\Re(a))/\mu}}\right)
    \right)
\end{equation}
for
\begin{equation} \label{kqb}
     I^*_s(N) := \int_{z'_1}^{z'_2} e^{-w} P_s(w) (z-z_0)^{s+a-1} \, dz.
\end{equation}
Similarly to Lemma \ref{min3} and using \eqref{xper}, we change variables to $w$ in \eqref{kqb} to produce
\begin{equation}  \label{kqc}
I^*_s(N) = \frac{e^{2\pi i k_1 (s+a)/\mu}}{\mu \cdot(N p_0)^{(s+a)/\mu}} \int_{N|p_0|(\rho/2)^\mu}^{N|p_0|(\rho/2)^\mu} e^{-w} P_s(w) w^{(s+a)/\mu-1}\, dw.
\end{equation}
 The path of integration in \eqref{kqc} starts and ends at the positive real number $T:=N|p_0|(\rho/2)^\mu$, circling the origin $k_2-k_1$ times. The value of $w^{(s+a)/\mu-1}$ in \eqref{kqc} is the principal power value at the beginning of the integration path and $\exp(2\pi i(k_2-k_1)(s+a)/\mu)$ times this value at the end  of the integration path.


\begin{lemma} \label{tgrt}
 If $(s+a)/\mu \in \Z_{\lqs 0}$ then
 \begin{equation*}
   I^*_s(N) =  2\pi i(k_2-k_1) \frac{(-1)^{(s+a)/\mu}}{|(s+a)/\mu|!} \frac{\alpha_s}{N^{(s+a)/\mu}}.
 \end{equation*}
\end{lemma}
\begin{proof}
Let $m:=(s+a)/\mu$ and the integral in \eqref{kqc} is
\begin{equation} \label{kqd}
  \int_{T}^{T} e^{-w} P_s(w) w^{m-1}\, dw.
\end{equation}
When $m \in \Z_{\lqs 0}$, the integrand has a pole with residue
\begin{equation*}
  \sum_{\ell=0}^{|m|} c_{s,\ell} \frac{(-1)^{|m|-\ell}}{(|m|-\ell)!}
\end{equation*}
where $c_{s,\ell}$ is the coefficient of $(z-z_0)^s$ in the Taylor expansion of $q(z)\phi(z)^\ell/\ell!$ about $z_0$ as in \eqref{polyP}. Therefore \eqref{kqd} equals the coefficient of $(z-z_0)^s$ in
\begin{align*}
  2\pi i(k_2-k_1) q(z)\sum_{\ell=0}^{|m|} \frac{(-1)^{|m|-\ell}}{(|m|-\ell)!} \frac{\phi(z)^\ell}{\ell!}  & =  2\pi i(k_2-k_1) \frac{(-1)^{|m|}}{|m|!} q(z)\sum_{\ell=0}^{|m|}\binom{|m|}{\ell} (-\phi(z))^\ell  \\
   &  =  2\pi i(k_2-k_1) \frac{(-1)^{m}}{|m|!} q(z) (1-\phi(z))^{-m}.
\end{align*}
Putting this value into \eqref{kqc} and comparing with \eqref{as2} completes the proof of the lemma.
\end{proof}

\begin{lemma} \label{tgrt2}
 If $(s+a)/\mu \not\in \Z_{\lqs 0}$ then, for $\varepsilon'' >0$,
 \begin{equation*}
   I^*_s(N) = \G\left(\frac{s+a}{\mu}\right) \frac{\alpha_s  \left( e^{2\pi i {k_2}(s+a)/\mu}- e^{2\pi i {k_1}(s+a)/\mu}\right)}{N^{(s+a)/\mu}} +O(K_q e^{-\varepsilon'' N}).
 \end{equation*}
\end{lemma}
\begin{proof}
Let $\mathcal{H}_T$ be the path that starts at infinity, follows the positive real line to $T$, circles the origin $k_2-k_1$ times and then returns from $T$ to its starting point at infinity.  We need the simple extension of \eqref{gam} given by
\begin{equation}\label{gam2}
    \int_T^{\infty} e^{-w} w^{d-1} \,dw \lqs  e^{-T/2} \times \begin{cases} 2^d \G(d) & \text{if } d\gqs 1 \\
    2T^{d-1}  & \text{if } d\lqs 1
    \end{cases}
\end{equation}
for $T>0$. Then arguing as at the start of Proposition \ref{min4} shows that the integral in \eqref{kqc} satisfies
\begin{equation*}
  \int_{T}^{T} e^{-w} P_s(w) w^{(s+a)/\mu-1}\, dw = \int_{\mathcal{H}_T} e^{-w} P_s(w) w^{(s+a)/\mu-1}\, dw +O(K_q e^{-\varepsilon'' N})
\end{equation*}
for $T=N|p_0|(\rho/2)^\mu$ and $\varepsilon''=|p_0|(\rho/2)^\mu/2$.

Now we claim that
\begin{equation}\label{keee}
 \frac{e^{2\pi i k_1 (s+a)/\mu}}{\mu \cdot(N p_0)^{(s+a)/\mu}}
  \int_{\mathcal{H}_T} e^{-w} P_s(w) w^{(s+a)/\mu-1}\, dw = \G\left(\frac{s+a}{\mu}\right) \frac{\alpha_s  \left( e^{2\pi i {k_2}(s+a)/\mu}- e^{2\pi i {k_1}(s+a)/\mu}\right)}{N^{(s+a)/\mu}}
\end{equation}
for all $T>0$ and for all $a\in \C$ with $(s+a)/\mu \not\in \Z_{\lqs 0}$. If $s+\Re(a)>0$ then we may let $T\to 0$ and evaluate the integrals along $\R_{\gqs 0}$ as in the second half of Lemma \ref{min4}. This proves \eqref{keee} for $a$ in a right half plane. However, the left side of \eqref{keee} is a holomorphic function of $a$ for all $a\in \C$. The right side  of \eqref{keee} is also holomorphic for all $a\in \C$ except that the $\G$ function has poles at the non-positive integers. Hence, the holomorphic functions on each side \eqref{keee} must agree for all $a \in \C$, except for the non-positive integers, and the lemma follows.
\end{proof}

We note that, since the left side of \eqref{keee} is  holomorphic in $a$, taking a limit in $a$  so that $(s+a)/\mu$ approaches a non-positive integer on the right side of \eqref{keee} can also be used to prove Lemma \ref{tgrt}.

With \eqref{kqa} and Lemmas \ref{tgrt} and \ref{tgrt2}, we have proved Theorem \ref{m6} at least for $S$ sufficiently large to satisfy the conditions before Lemma \ref{mintt}. The terms $\alpha_s/N^{(s+a)/\mu}$ are $O(K_q /N^{(s+\Re(a))/\mu})$, (see Proposition \ref{albnd} below), and so we obtain the theorem for all $S \in \Z_{\gqs 0}$.
\end{proof}

\subsection{Further generalizations} \label{laplace}
The main results of Theorems \ref{il}, \ref{m4} and \ref{m6} may be extended in different directions:
\begin{itemize}
  \item The case where the contour of integration $\cc$ has an endpoint at infinity can easily be handled if the part of the integral near infinity has a bound such as $O(e^{-\varepsilon N})$.
  \item It is possible to let $\mu$ in \eqref{f} be a positive real number instead of just a positive integer - see for example \cite[Thm. 6.1, p. 125]{Ol}. Of course $p(z)$ will no longer be holomorphic in a neighborhood of $z_0$ if $\mu$ is not an integer.
  \item With extra conditions, as described in \cite[Thm. 4, p. 105]{Wo89} or \cite[Thm. 6.1, p. 125]{Ol}, we may allow $N$ to approach infinity in a sector in $\C$
  \item Laplace's method, originating with Laplace in the 18th century, gives the main term of the asymptotics of \eqref{i(n)} where $\cc$ is an interval on the real line and $p(z)$ and $q(z)$ are real-valued. It is assumed that there exists  a unique maximum of $p(z)$ on $\cc$ (at $z=z_0$, say) along with the weak conditions that $p(z)$ is differentiable with $p'(z)$ and $q(z)$ continuous; see for example \cite[Thm. 7.1, p. 81]{Ol} for the precise statement. When $p(z),$ $q(z)$ have series expansions in a neighborhood of $z_0$ then as in \cite[Thm. 8.1, p. 86]{Ol}, the full asymptotic expansion of \eqref{i(n)} can be given.
      If $p(z)$ and $q(z)$ are restrictions of holomorphic functions on a domain containing $\cc$, then Perron's method may be applied to obtain the same result since $z_0$ is necessarily a saddle-point with  steepest descent angles lying on the real line.
  \item In Section VIII of \cite{Fl09} a general type of saddle-point algorithm is provided to attempt to find the asymptotics as $N\to \infty$ of integrals $\int_\cc F(z)\, dz$ where $F(z)$ depends in some way on $N$. 
\end{itemize}

\section{More formulas for $\alpha_s$}  \label{formul}

If we know the order of vanishing of $q(z)$ at $z=z_0$ then we can say which of the first numbers $\alpha_0,$ $\alpha_1, \ldots$ in Theorems \ref{il} or \ref{m4} are zero.

\begin{prop}
Let the order of vanishing of $q(z)$ at $z=z_0$ be $m$ and write $q(z)=(z-z_0)^{m}\psi(z)$ where $\psi(z_0) \neq 0$.
Then we have $\alpha_0 = \alpha_1 = \cdots = \alpha_{m-1}=0$ and $\alpha_{m}\neq 0$. Also, for $s \in \Z_{\gqs 0}$,
\begin{equation} \label{sjk}
   \alpha_{s+m} =  \frac{1}{\mu \cdot s!} p_0^{-(s+m+a)/\mu} \frac{d^s}{dz^s}\left\{\psi(z) \cdot \left( 1-\phi(z)\right)^{-(s+m+a)/\mu)} \right\}_{z=z_0}.
\end{equation}
\end{prop}
\begin{proof}
Replace $q(z)$ by $(z-z_0)^{m}\psi(z)$ in \eqref{as2}, and evaluate the derivative with Leibnitz's rule and the fact that
\begin{equation*}
    \left.\frac{d^s}{dz^s}(z-z_0)^{m}\right|_{z=z_0} = \begin{cases}
    m! & \text{if $\ s=m$}\\ 0 & \text{if $\ s\neq m$}.
    \end{cases}
\end{equation*}
It follows easily that $\alpha_s=0$ for $s\lqs m-1$ and that \eqref{sjk} holds. Also \eqref{sjk} implies that $\alpha_{m}$ takes the non-zero value $p_0^{-(m+a)/\mu}\psi(z_0)/\mu $.
\end{proof}

Therefore, in Theorems \ref{il} and \ref{m4} where $\cc$ starts at $z_0$, the main term of the asymptotic expansion has $s=m$ where $m$ is the order of vanishing of $q(z)$.

In Corollaries \ref{il2}, \ref{m5} and Theorem \ref{m6} where $\cc$ passes through $z_0$, the main  term of the asymptotic expansion may not be $s=m$, since the factor $e^{2\pi i {k_2}(s+a)/\mu}- e^{2\pi i {k_1}(s+a)/\mu}$  vanishes when $(k_2-k_1)(s+a)/\mu \in \Z$, and a calculation is required to find the first non-zero term. In some cases the terms $\alpha_s(e^{2\pi i {k_2}(s+a)/\mu}- e^{2\pi i {k_1}(s+a)/\mu})$ vanish for all $s$ and we do not obtain exact asymptotics with these results. This happens for example when $\mu=1$ and $a\in \Z$, or when $q(z)=p'(z)$.

As before, write
\begin{equation*}
    p(z)-p(z_0)=-\sum_{s=0}^\infty p_s (z-z_0)^{s+\mu}, \qquad q(z)=\sum_{s=0}^\infty q_s (z-z_0)^{s}.
\end{equation*}
The next result is due to Campbell,  Fr{\"o}man and Walles \cite[pp. 157-158]{CFW} and expresses $\alpha_s$ in terms of the coefficients $p_s$ and $q_s$. It requires
 the {\em partial ordinary Bell polynomials} which may be defined with the generating function
\begin{equation} \label{pobell2}
    \left( p_1 x +p_2 x^2+ p_3 x^3+ \cdots \right)^j = \sum_{i=j}^\infty \hat{B}_{i,j}(p_1, p_2, p_3, \dots) x^i.
\end{equation}
 It is straightforward to see they may also be given as
\begin{equation} \label{pobell}
    \hat{B}_{i,j}(p_1, p_2, p_3, \dots) = \sum_{\substack{1\ell_1+2 \ell_2+ 3\ell_3+\dots = i \\ \ell_1+ \ell_2+ \ell_3+\dots = j}}
    \frac{j!}{\ell_1! \ell_2! \ell_3! \cdots } p_1^{\ell_1} p_2^{\ell_2} p_3^{\ell_3} \cdots
\end{equation}
from \cite[Sect. 3.3]{Comtet} where the sum is over all possible $\ell_1$, $\ell_2$, $\ell_3, \dots \in \Z_{\gqs 0}$, or as
\begin{equation} \label{pobell3}
    \hat{B}_{i,j}(p_1, p_2, p_3, \dots)= \sum_{n_1+n_2+\dots + n_j = i}
    p_{n_1}p_{n_2} \cdots p_{n_j}
\end{equation}
for $j \gqs 1$ from \cite[p. 156]{CFW} where the sum is over all possible $n_1$, $n_2,  \dots \in \Z_{\gqs 1}$.
See \cite[Sect. 3.3]{Comtet} for more information on Bell polynomials, including their recurrence relations.

\begin{prop} \label{wojf}
For $\alpha_s$ defined in \eqref{as2},
\begin{equation} \label{hjw}
    \alpha_s = \frac{1}{\mu} p_0^{-(s+a)/\mu} \sum_{i=0}^s q_{s-i}\sum_{j=0}^i  \binom{-(s+a)/\mu}{j}  \hat B_{i,j}\left(\frac{p_1}{p_0},\frac{p_2}{p_0},\cdots\right).
\end{equation}
\end{prop}
\begin{proof}
We have
\begin{align}
     \left( 1-\phi(z)\right)^{-(s+a)/\mu} & =
    \left( 1+\sum_{s=1}^\infty \frac{p_s}{p_0} (z-z_0)^{s}\right)^{-(s+a)/\mu} \notag \\
    & = \sum_{j=0}^\infty \binom{-(s+a)/\mu}{j} \left(\sum_{s=1}^\infty \frac{p_s}{p_0} (z-z_0)^{s}\right)^j \label{veni}\\
    & = \sum_{j=0}^\infty  \binom{-(s+a)/\mu}{j} \sum_{i=j}^\infty \hat B_{i,j}\left(\frac{p_1}{p_0},\frac{p_2}{p_0},\cdots\right)(z-z_0)^i. \notag
\end{align}
Therefore the coefficient of $(z-z_0)^s$ in $q(z) \left( 1-\phi(z)\right)^{-(s+a)/\mu}$ is
\begin{equation*}
    \sum_{i=0}^s q_{s-i}\sum_{j=0}^i  \binom{-(s+a)/\mu}{j}  \hat B_{i,j}\left(\frac{p_1}{p_0},\frac{p_2}{p_0},\cdots\right)
\end{equation*}
and the result follows.
\end{proof}

With $a=1$, the first cases are $\alpha_0= p_0^{-1/\mu} q_0/\mu$ and
\begin{equation*}
    \alpha_1=\frac{1}{\mu} p_0^{-2/\mu}\left( -\frac{2p_1 q_0}{\mu p_0}+q_1\right),\qquad
  \alpha_2=\frac{1}{\mu} p_0^{-3/\mu}\left( \frac{3(1+3/\mu)p_1^2 q_0}{2\mu p_0^2}-\frac{3(p_2 q_0+p_1 q_1)}{\mu p_0}+q_2\right).
\end{equation*}
Moving $p_0$ out of the sum in \eqref{veni} gives the slightly different formulation
\begin{equation} \label{hjw2}
    \alpha_s = \frac{1}{\mu} p_0^{-(s+a)/\mu} \sum_{i=0}^s q_{s-i}\sum_{j=0}^i p_0^{-j} \binom{-(s+a)/\mu}{j}  \hat B_{i,j}(p_1,p_2,\cdots).
\end{equation}
Wojdylo \cite{Woj06} rediscovered the formula \eqref{hjw} in the context of Laplace's method, though his proof  seems incomplete; the form of \cite[Eq. (2.34)]{Woj06} needs to be justified.
A comparison of the schemes to give $\alpha_s$ explicitly  in \cite{Pe17,deB,Din,CFW,Woj06} is discussed in the Appendix of
\cite{LP2011}. See also \cite{Nem13}.

We finish this section with a new bound for these expansion coefficients.

\begin{prop} \label{albnd}
With  Assumptions \ref{ma0} and $\alpha_s$ defined in \eqref{as2},
\begin{equation} \label{albnd2}
    \alpha_{s} = O\bigl( K^*_q \cdot C^s \bigr) \quad \text{for} \quad s \in \Z_{\gqs 0}
\end{equation}
where $K^*_q$ is a bound for  $|q(z)|$ on $\nb$. The positive constant $C$ and the implied constant in \eqref{albnd2} are both independent of $q$ and $s$.
\end{prop}
\begin{proof}
The result follows from \eqref{dgf2} with $C$ taken as the reciprocal of the radius of $\mathcal D_\tau$.
\end{proof}

\section{Applications} \label{app}
The next examples illustrate how to apply Perron's method. Given an integral depending on a parameter $N$ going to infinity, the first task is to try to get it into the form \eqref{i(n)}, perhaps with a change of variables. We are free to move  the path of integration $\cc$ continuously wherever the integrand is holomorphic. If we can ensure that $\Re(p(z))$ is maximized at an endpoint then Theorems \ref{il} or \ref{m4} may be applied. Otherwise we move $\cc$ to pass through saddle-points and employ Corollaries \ref{il2}, \ref{ilf}, \ref{m5} or Theorem \ref{m6}.

\subsection{Gamma function asymptotics} \label{simex}

The standard example, see e.g. \cite[Sect. 5]{Pe17}, is the important gamma function.
For $N>0$ we have
\begin{equation*}
  \G(N+1)=\int_0^\infty e^{-t} t^N \, dt = N^{N+1}\int_0^\infty e^{N(-z+\log z)}  \, dz
\end{equation*}
with the change of variables $t=Nz$. Fitting the last integral into \eqref{i(n)} and Assumptions \ref{ma0}, write  $q(z)=1$ and $p(z)=-z+\log z$ with $p'(z)=-1+1/z$. This shows there is a saddle-point at $z_0=1$.
Close to $z=1$ we have the expansion $-\log z = (1-z)+(1-z)^2/2+(1-z)^3/3+\cdots$, so the range of integration can be restricted to $[1/2,3/2]$, say, and it is easy to see that the remaining integral will be too small to affect the result.

 Hence, for $|z-1|\lqs 1/2$, $p(z)$ equals
\begin{equation*}
  p(1)-p_0 (z-1)^\mu(1-\phi(z)) = -1-\frac 12 (z-1)^2\left(1-\frac 23(z-1) + \frac 24 (z-1)^2 + \cdots\right)
\end{equation*}
so that
$p_0=1/2$, $\omega_0=0$, $\mu=2$  and $p_s/p_0=(-1)^s 2/(s+2)$. The steepest descent angles are $\theta_\ell = \pi\ell$.  The assumptions of Corollary \ref{ilf} hold (with $k=0$)
and on simplifying it shows
\begin{equation*}
  \G(N+1)= \sqrt{2\pi N}\left(\frac{N}{e}\right)^N \left(1+\frac{\gamma_1}{N}+\frac{\gamma_2}{N^2}+ \cdots + \frac{\gamma_{k-1}}{N^{k-1}} +O\left(\frac{1}{N^k}\right)\right)
\end{equation*}
for, by Proposition \ref{wojf} (with $a=1$),
\begin{equation} \label{g!}
  \gamma_m = \frac{(2m)!}{m! 2^m} \sum_{j=0}^{2m}  \binom{-m-1/2}{j} \hat{B}_{2m,j}\left(-\frac{2}{3},\frac{2}{4},-\frac{2}{5},\frac{2}{6},-\frac{2}{7}, \cdots \right).
\end{equation}
The first coefficients are
$\g_1=
1/12,$ $\g_2=1/288,$ $\g_3=-139/51840$ as Laplace already knew.  See  \cite[Example 1]{Nem13} for different treatments of \eqref{g!}. Approximations to the gamma function are still an interesting and active area of research as shown in \cite{Ch13}.

\subsection{The equation of the center} \label{eqcent}
In Kepler's theory of motion, the planets orbit the sun in ellipses of eccentricity $\varepsilon$ with the sun at one focus. The true anomaly $\nu$ is the angle  made from this focus and may be compared with the angle $M$ (the mean anomaly) made if the planet were in uniform circular motion, with the same period, about the mid point of the foci. These quantities are related by Kepler's equations
\begin{equation*}
  \cos \nu = \frac{\cos E - \varepsilon}{1-\varepsilon \cos E}, \qquad M=E-\varepsilon \sin E
\end{equation*}
for an intermediate quantity $E$, called the eccentric anomaly.
The  {\em equation of the center} refers to different ways to relate $\nu$ to $M$ directly. An important way is through the Fourier expansion
\begin{equation} \label{cen}
\nu - M = \sum_{n=1}^\infty C_n \sin(nM) \qquad \text{for} \qquad C_n = \frac{\sqrt{1-\varepsilon^2}}{\pi n}\int_{-\pi}^{\pi} \frac{e^{n i(z- \varepsilon\sin z)}}{1-\varepsilon\cos z} \,dz,
\end{equation}
as derived in  \cite[pp. 210--212]{Ba}, for example.
 The integral appearing in \eqref{cen}  is the one from the introduction, \eqref{vare}. Before working on the asymptotics of \eqref{vare} we take a simpler case.

The integral
\begin{equation}\label{51}
   \int_{-\pi}^{\pi} e^{N i(z- \sin z)} \,dz \qquad (N \in \Z_{\gqs 1})
\end{equation}
is studied in \cite{Bu14}, \cite{Pe17}.
Fitting it to the assumptions of Corollary \ref{il2}, we have $q(z)=1$ and $p(z)=i(z- \sin z)$ with $p'(z)=i(1-\cos z)$. This shows there is a saddle-point at $z_0=0$ and writing
\begin{equation*}
  p(z)=0-p_0 z^\mu(1-\phi(z)) = -\left(\frac{-i}{3!}\right)z^3\left(1-\frac{3!}{5!} z^2 + \cdots\right)
\end{equation*}
means that
$p_0=-i/6$, $\omega_0=-\pi/2$, $\mu=3$  and $1-\phi(z)=6(z-\sin z)/z^3$. The steepest descent angles are $\theta_\ell = \pi/6+2\pi\ell/3$ as shown in Figure \ref{spokesfig}. We change the path of integration to
\begin{equation} \label{eqint}
  \int_{-\pi}^{-\pi+\pi i/\sqrt{3}}+ \int_{-\pi+\pi i/\sqrt{3}}^0 +  \int_0^{\pi+\pi i/\sqrt{3}} +  \int_{\pi+\pi i/\sqrt{3}}^{\pi}
\end{equation}
so that $0$ is approached along the line with angle $\theta_{k_1}$ for
 $k_1=1$ and, on leaving $0$, the line with angle $\theta_{k_2}$ for $k_2=0$ is followed. The integrals along the vertical lines cancel since the integrand has period $2\pi$.
 We have
 \begin{equation*}
   \Re(p(t e^{i \theta_0})) = \Re(p(t e^{i \theta_1})) = f(t) \quad \text{for} \quad f(t):=\cos(\sqrt{3}t/2)\sinh(t/2)-t/2
 \end{equation*}
 with $f(0)=0$. To confirm condition \eqref{c1x} we need to show that $f(t)<0$ for $0<t\lqs 2\pi/\sqrt{3}$.
 One approach is to first note that
 \begin{equation*}
   f'''(t) = -\cos(\sqrt{3}t/2) \cosh(t/2).
 \end{equation*}
 Hence $f''(t)$ is decreasing on $[0,\pi/\sqrt{3})$ and increasing on $[\pi/\sqrt{3},2\pi/\sqrt{3}]$. As $f''(0)=0$ and $f''(2\pi/\sqrt{3})$ is positive, this means that $f''(t)$ is negative on an interval $(0,c)$ and  positive on $(c,2\pi/\sqrt{3}]$ for some $c$.  We see that $f'(t)$ decreases from $f'(0)=0$ and then increases from $t=c$ to $f'(2\pi/\sqrt{3})$ which is $<0$. Therefore $f'(t)$ is negative on $(0,2\pi/\sqrt{3}]$ and so $f(t)$ is decreasing in this range as we wanted.

 Write
\begin{equation*}
  \alpha_s  = \frac{e^{\pi i(s+1)/6} \cdot 6^{(s+1)/3} \cdot d(s)}{3} \quad \text{for} \quad d(s)= \frac{1}{s!} \frac{d^s}{dz^s}\left\{ \left(\frac{6(z-\sin z)}{z^3}\right)^{-(s+1)/3}\right\}_{z=0}.
\end{equation*}
 Also, by Proposition \ref{wojf},
\begin{equation} \label{dss}
  d(s)=\sum_{j=0}^s \binom{-(s+1)/3}{j} \hat{B}_{s,j}\left(0,-\frac{3!}{5!},0,\frac{3!}{7!},0,-\frac{3!}{9!}, \cdots \right)
\end{equation}
and computations yield for example
\begin{equation*}
  d(0)=1, \ d(2)= \frac{1}{20}, \ d(4)= \frac{1}{280}, \ d(6)= \frac{1}{3600}, \ d(8)= \frac{387}{17248000}
\end{equation*}
with $d(s)=0$ for $s$ odd. Then by  Corollary \ref{il2}, for an implied constant depending only on $S$,
\begin{align}
  \int_{-\pi}^{\pi} e^{N i(z- \sin z)} \,dz & =\sum_{s=0}^{S-1}  \G\left(\frac{s+1}{3}\right) \frac{\alpha_s \left(1-e^{2\pi i(s+1)/3}\right)}{N^{(s+1)/3}}+O\left(\frac{1}{N^{(S+1)/3}}\right)\notag\\
  & = \frac 23 \sum_{s=0}^{S-1} \cos\left(\frac{\pi(s+1)}{6}\right) \G\left(\frac{s+1}{3}\right)d(s) \left(\frac 6N \right)^{(s+1)/3}+O\left(\frac{1}{N^{(S+1)/3}}\right).\label{trui}
\end{align}
We can obtain non-zero terms in the sum only for $s\equiv 0,4 \bmod 6$. Formulas \eqref{dss} and \eqref{trui}  give the complete asymptotic expansion of the integral \eqref{51}.
With $S=10$ for example,
\begin{multline} \label{woh}
  \int_{-\pi}^{\pi} e^{N i(z- \sin z)} \,dz = \frac 23 \cos\left(\frac{\pi}{6}\right)  \G\left(\frac{1}{3}\right)\left(\frac 6N \right)^{1/3}
  + \frac 1{420} \cos\left(\frac{5\pi}{6}\right)  \G\left(\frac{5}{3}\right)\left(\frac 6N \right)^{5/3}\\
  + \frac 1{5400} \cos\left(\frac{7\pi}{6}\right)  \G\left(\frac{7}{3}\right)\left(\frac 6N \right)^{7/3}
  +O\left(\frac{1}{N^{11/3}}\right)
\end{multline}
which is equivalent to \cite[Eq. (53)]{Pe17}. When $N=50$, for instance, the integral in \eqref{woh} is approximately $\underline{0.76283538}2546$ with the underlined digits indicating the agreement with the right side of \eqref{woh}. All the numerical calculations in this paper were carried out using Mathematica.

\subsection{Asymptotics of the true anomaly Fourier coefficient $C_N$} \label{egg}
We now turn to the integral
\begin{equation}\label{vare2}
  \int_{-\pi}^{\pi} \frac{e^{N i(z- \varepsilon\sin z)}}{1-\varepsilon\cos z} \,dz
   \qquad (N \in \Z_{\gqs 1}, 0<\varepsilon<1),
\end{equation}
appearing in the equation of the center \eqref{cen}, and the main motivation of the papers \cite{Bu14,Pe17}. We initially follow \cite[pp. 210-214]{Pe17} and then go more deeply into the combinatorics of the expansion coefficients.

Set $p(z)=i(z- \varepsilon\sin z)$ and so $p'(z)=i(1-\varepsilon\cos z)$. It is convenient to define
\begin{equation*}
  \g := \frac{1+\sqrt{1-\varepsilon^2}}{\varepsilon}>1.
\end{equation*}
Then $p'(z)=0$ for $z$ taking the two values $\pm i\log \g$ and we choose $z_0$ to be $i\log \g$. Our computations will show that this is the correct choice. Expanding about this saddle-point gives
\begin{equation} \label{deer}
  \varepsilon \sin(z+z_0) = \sin z+i\sqrt{1-\varepsilon^2} \cos z, \qquad
  \varepsilon \cos(z+z_0) = \cos z-i\sqrt{1-\varepsilon^2} \sin z
\end{equation}
as in \cite[p. 212]{Pe17}. Hence
\begin{align}
  p(z)-p(z_0) & = i(z-z_0)-i\sin(z-z_0) + \sqrt{1-\varepsilon^2}(-1+\cos(z-z_0)) \label{xms} \\
   & = -\frac{\sqrt{1-\varepsilon^2}}{2}(z-z_0)^2\left(1-\frac{2i}{3!\sqrt{1-\varepsilon^2}}(z-z_0)
   -\frac{2}{4!}(z-z_0)^2+\cdots\right) \notag
\end{align}
which implies that $p_0=\sqrt{1-\varepsilon^2}/2$, $\mu=2$ and the steepest descent angles are $\theta_\ell = \pi\ell$. Clearly, $1/(1-\varepsilon\cos z)$ has a simple pole at $z=z_0$, so we let  $q(z)=(z-z_0)/(1-\varepsilon\cos z)$ with $a=0$ and will be applying Theorem \ref{m6}.

The contour of integration should therefore be moved from the real line and go vertically from $-\pi$ to $-\pi+i\log \g$. The path then  approaches $z_0$ along the steepest descent angle $\theta_{k_1}$ for $k_1=1$, circles below $z_0$ and leaves along the angle $\theta_{k_2}$ for $k_2=2$. After reaching $\pi+i\log \g$ it then moves vertically to $\pi$. The integrals on the vertical paths cancel since the integrand has period $2\pi$. For $t\in \R$ we see by \eqref{xms}
that $\Re(p(t+z_0)-p(z_0))=\sqrt{1-\varepsilon^2}(-1+\cos(t))$ and so
 $\Re(p(z)) < \Re(p(z_0))$ for $z$ on the horizontal part of the contour and the conditions for Theorem \ref{m6} are satisfied. (The other saddle-point, $-i\log \g$, has vertical steepest descent lines and so we cannot use it in a similar treatment.)

 Writing $w$ for $z-z_0$ we obtain by \eqref{deer}
\begin{equation*}
  q(z)=\frac{w}{1-\cos w+i \sqrt{1-\varepsilon^2} \sin w} = \frac{2\g}{\varepsilon}\cdot \frac{w}{(e^{iw}-1)} \cdot \frac{1}{(\g^2 e^{-iw}-1)}.
\end{equation*}
We have the expansions
\begin{equation}\label{expsx}
  \frac{z}{e^z-1} = \sum_{m=0}^\infty \frac{B_m}{m!} z^m, \qquad \frac{1}{\xi e^z-1} = \sum_{m=0}^\infty \frac{\beta_{m+1}(\xi)}{(m+1)!} z^m \qquad (\xi \neq 1)
\end{equation}
for the Bernoulli numbers $B_m$ and the coefficients
\begin{equation} \label{glai}
  \beta_{m}(\xi) = (-1)^{m-1} m \sum_{j=1}^{m}  \stirb{m}{j} \frac{(j-1)!}{(\xi - 1)^j} \qquad (\xi \neq 1)
\end{equation}
where $\stirb{m}{j}$ is the Stirling number,  denoting the number of ways to partition a set of size $m$ into $j$  non-empty subsets. See \cite[Prop. 3.2]{OS15} for the formula \eqref{glai} which is similar to a result of Glaisher. Then
\begin{equation*}
   q(z) = \sum_{s=0}^\infty q_s w^s =  \frac{-2\g \cdot i}{\varepsilon} \left(\sum_{m=0}^\infty \frac{B_m}{m!}(i w)^m \right)\left(\sum_{n=0}^\infty \frac{\beta_{n+1}(\g^2)}{(n+1)!}(-i w)^n \right)
\end{equation*}
and we obtain the expression
\begin{equation} \label{xkp}
  q_s = \frac{-2\g \cdot i^{s+1}}{\varepsilon} \sum_{n=0}^s (-1)^n \frac{\beta_{n+1}(\g^2) B_{s-n}}{(n+1)!(s-n)!}.
\end{equation}
With  Proposition \ref{wojf} we may write $\alpha_s=d(s)/(2p_0^{s/2})$ for
\begin{equation} \label{dxm}
  d(s)=\sum_{i=0}^s q_{s-i} \sum_{j=0}^i \binom{-s/2}{j} \hat{B}_{i,j}\left(-\frac{2i}{3!\sqrt{1-\varepsilon^2}},-\frac{2}{4!},
  \frac{2i}{5!\sqrt{1-\varepsilon^2}},\frac{2}{6!}, \cdots \right)
\end{equation}
where the arguments in the above Bell polynomial are
\begin{equation*}
  \frac{p_s}{p_0}=\frac{2\cdot i^s}{(s+2)!} \times \begin{cases}
                                                     1, & \mbox{if $s$ is even}  \\
                                                     -1/\sqrt{1-\varepsilon^2}, & \mbox{if $s$ is odd.}
                                                   \end{cases}
\end{equation*}
A short calculation with \eqref{deer} shows
\begin{equation*}
  e^{p(z_0)} = e^{\sqrt{1-\varepsilon^2}}/\g <1.
\end{equation*}
Putting everything together, and using the last line in the statement of Theorem \ref{m6} for the $s=0$ term, we obtain
\begin{multline}\label{46xb}
    \int_{-\pi}^{\pi} \frac{e^{N i(z- \varepsilon\sin z)}}{1-\varepsilon\cos z} \,dz\\
   = \left(\frac{e^{\sqrt{1-\varepsilon^2}}}{\g}\right)^N\left(
   \frac{\pi}{\sqrt{1-\varepsilon^2}}+\sum_{1\lqs s \lqs S-1, \ s \text{ odd}} \G(s/2) \cdot d(s) \left( \frac{2}{\sqrt{1-\varepsilon^2}N}\right)^{s/2} + O\left(\frac{1}{N^{S/2}}\right)
   \right)
\end{multline}
which, along with \eqref{xkp} and \eqref{dxm}, gives the complete asymptotic expansion.
Computing the first values of $d(s)$, for $s$ odd, we observe that they take the form $f_s(\varepsilon^2)/ (1-\varepsilon^2)^{(s+1)/2}$ for $f_s$ a polynomial with rational coefficients and degree $(s-1)/2$. For instance
\begin{equation*}
  f_1(x)=2/3, \quad f_3(x)=-(46+189x)/540, \quad f_5(x)=(92+6228x+4887x^2)/36288.
\end{equation*}
It would be interesting to prove that this form always holds. With $S=5$ we find
\begin{multline}\label{46xb2}
    \int_{-\pi}^{\pi} \frac{e^{N i(z- \varepsilon\sin z)}}{1-\varepsilon\cos z} \,dz
   = \left(\frac{e^{\sqrt{1-\varepsilon^2}}}{\g}\right)^N\left(
   \frac{\pi}{\sqrt{1-\varepsilon^2}}+ \G(1/2)\frac{2}{3(1-\varepsilon^2)} \left( \frac{2}{\sqrt{1-\varepsilon^2}N}\right)^{1/2}\right. \\
   \left. - \G(3/2)\frac{46+189\varepsilon^2}{540(1-\varepsilon^2)^2}  \left( \frac{2}{\sqrt{1-\varepsilon^2}N}\right)^{3/2} + O\left(\frac{1}{N^{5/2}}\right)
   \right)
\end{multline}
which is equivalent to \cite[Eq. (45)]{Pe17}. When $N=50$ and $\varepsilon=2/5$, for example, the integral in \eqref{46xb2} is $\approx \underline{2.8171}413884\times 10^{-14}$ with the underlined digits indicating the agreement with the right side of \eqref{46xb2}. Taking $S=13$, i.e. using the first $7$ terms in the expansion \eqref{46xb}, yields the agreement $\underline{2.817141388}4\times 10^{-14}$.

As a referee noted, the method of steepest descent for this example requires moving the contour of integration to a more complicated path near $z_0$ than the horizontal line above. It requires part of the path described by the equation $\cosh(y)=x/(\varepsilon \sin(x))$ for $z=x+iy$. This is where $\Im(p(z)-p(z_0))=0$.

\subsection{The case $\varepsilon = 1$}  \label{egg2}
Taking $\varepsilon = 1$ in \eqref{vare2} produces the integral
\begin{equation}\label{46}
   \int_{-\pi}^{\pi} \frac{e^{N i(z- \sin z)}}{1-\cos z} \,dz \qquad (N \in \Z_{\gqs 1})
\end{equation}
which is studied in example 4 of \cite{Pe17}. This would correspond to a parabolic orbit if \eqref{cen} were valid for $\varepsilon = 1$. The path of integration in \eqref{46} must avoid the double pole at $z=0$ in order to converge. The expansion of the integrand at $z=0$ begins $2/z^2 + 1/6 + (N i z)/3 + z^2/120 + \cdots$, implying the residue at $z=0$ is zero. Since the integrand has period $2\pi$, all the residues are zero and so the integral is completely independent of any pole-avoiding path of integration  from $-\pi$ to $\pi$.

The function $p(z)$ is the same as in Section \ref{eqcent}, but now $q(z)=z^2/(1-\cos z)$ and $a=-1$. We will use Theorem \ref{m6} and so the path of integration \eqref{eqint} must be adjusted to circle at a small radius about the pole at $z_0=0$.
 Then
\begin{equation*}
  \alpha_s  = \frac{e^{\pi i(s-1)/6} \cdot 6^{(s-1)/3} \cdot d^*(s)}{3} \quad \text{for} \quad d^*(s)= \frac{1}{s!} \frac{d^s}{dz^s}\left\{ \frac{z^2}{1-\cos z} \left(\frac{6(z-\sin z)}{z^3}\right)^{-(s-1)/3}\right\}_{z=0}.
\end{equation*}
We have
\begin{equation*}
   \sum_{s=0}^\infty q_s z^s = \frac{z^2}{1-\cos z} = 2\left(\frac{i z}{e^{i z}-1}\right)\left(\frac{-i z}{e^{-i z}-1}\right) = 2 \left(\sum_{m=0}^\infty \frac{B_m}{m!}(i z)^m \right)\left(\sum_{n=0}^\infty \frac{B_n}{n!}(-i z)^n \right).
\end{equation*}
It follows that $q_s$ is $0$ for odd $s$ and for $s$ even
\begin{equation} \label{dss2x}
  q_s = 2 (-1)^{s/2} \sum_{n=0}^s (-1)^n \frac{B_n B_{s-n}}{n!(s-n)!}.
\end{equation}
Proposition \ref{wojf} tells us
\begin{equation} \label{dss2}
  d^*(s)=\sum_{i=0}^s q_{s-i} \sum_{j=0}^i \binom{-(s-1)/3}{j} \hat{B}_{i,j}\left(0,-\frac{3!}{5!},0,\frac{3!}{7!},0,-\frac{3!}{9!}, \cdots \right)
\end{equation}
and computations yield for example
\begin{equation*}
  d^*(0)=2, \ d^*(2)= \frac{1}{5}, \ d^*(4)= \frac{27}{1400}, \ d^*(6)= \frac{23}{12600}, \ d^*(8)= \frac{947}{5544000}
\end{equation*}
with $d^*(s)=0$ for $s$ odd. Then, for an implied constant depending only on $S$,
\begin{equation}
  \int_{-\pi}^{\pi} \frac{e^{N i(z- \sin z)}}{1-\cos z} \,dz  = \frac 23 \sum_{s=0}^{S-1} \cos\left(\frac{\pi(s-1)}{6}\right) \G\left(\frac{s-1}{3}\right)d^*(s) \left(\frac 6N \right)^{(s-1)/3}+O\left(\frac{1}{N^{(S-1)/3}}\right).\label{trui2}
\end{equation}
We can obtain non-zero terms in the sum only for $s\equiv 0,2 \bmod 6$. The term with $s=1$ needs the formula from the last line of the statement of Theorem \ref{m6}, but in any case vanishes since $d^*(1)=0$. Formulas \eqref{dss2x}, \eqref{dss2} and \eqref{trui2}  give the complete asymptotic expansion of the integral \eqref{46}. Taking $S=8$ for example,
\begin{multline} \label{woh2}
  \int_{-\pi}^{\pi} \frac{e^{N i(z- \sin z)}}{1-\cos z} \,dz = \frac 43 \cos\left(\frac{-\pi}{6}\right)  \G\left(\frac{-1}{3}\right)\left(\frac N6 \right)^{1/3}
  + \frac 2{15} \cos\left(\frac{\pi}{6}\right)  \G\left(\frac{1}{3}\right)\left(\frac 6N \right)^{1/3}\\
  + \frac{23}{18900} \cos\left(\frac{5\pi}{6}\right)  \G\left(\frac{5}{3}\right)\left(\frac 6N \right)^{5/3}
  +O\left(\frac{1}{N^{7/3}}\right)
\end{multline}
with the first two terms of this expansion given in \cite[Eq. (50)]{Pe17}. When $N=50$ the integral in \eqref{woh2} is $\approx \underline{-9.35758}5773084$ and the underlined digits show the agreement with the right hand side.

\section{The asymptotics of Sylvester waves} \label{syl}
In this section we give an application of Perron's method to number theory. Let $p(n)$ be the number of  partitions of the positive integer $n$. This is the number of ways to write $n$ as a sum of non-increasing positive integers.
Also let $p_N(n)$ count the partitions of $n$ with at most $N$ summands.
Since the work of Cayley  and Sylvester in the nineteenth century, we know that
\begin{equation*} 
    p_N(n)=\sum_{k=1}^N W_k(N,n)
\end{equation*}
where each $W_k(N,n)$ may be expressed in terms of a  sequence of $k$ polynomials $w_{k,m}(N,x) \in \Q[x]$ for $0\lqs m \lqs k-1$. Write
\begin{equation} \label{wkv}
    W_k(N,n) = \bigl[ w_{k,0}(N,n), \ w_{k,1}(N,n), \ \ldots, \ w_{k,k-1}(N,n) \bigr],
\end{equation}
where the notation in \eqref{wkv} indicates that the value of $W_k(N,n)$ is given by one of the polynomials on the right and we select $w_{k,j}(N,n)$ when $n \equiv j \bmod k$. The degrees of the polynomials on the right of \eqref{wkv} are at most $\lfloor N/k \rfloor -1$.

For example,  with $N=3$ we have $p_3(n)=W_1(3,n)+ W_2(3,n) + W_3(3,n)$
where
\begin{align}
W_1(3,n) & = \left[6 n^2+36 n+47 \right]/72, \notag\\
    W_2(3,n) & = \left[1, \ -1 \right]/8,\notag\\
    W_3(3,n) & = \left[ 2, \ -1, \ -1 \right]/9. \notag
\end{align}
Sylvester called $W_k(N,n)$ the {\em $k$-th wave} and provided the formula
\begin{equation} \label{wave}
    W_k(N,n)=\res_{z=0} \sum_\rho \frac{\rho^n e^{n z}}{(1-\rho^{-1} e^{-z})(1-\rho^{-2} e^{-2z}) \cdots (1-\rho^{-N} e^{-Nz})}
\end{equation}
in \cite{Sy3}, where $\res_{z=0}$ indicates the coefficient of $1/z$ in the Laurent expansion about $0$, and the sum is over all primitive $k$-th roots of unity $\rho$. For a more detailed discussion of the above results with references, see Sections 1 and 2 of \cite{OSpsw}.

When $N=3$ it is clear that the first wave $W_1(3,n)$ will make the largest contribution to $p_3(n)$ for large $n$. Similarly, $p_N(n)  \sim W_1(N,n)$ for any fixed $N$ as $n \to \infty$. A more difficult question, which we answer for the first time in \cite{OSpsw}, is how the first waves $W_1(N,n)+W_2(N,n)+\cdots$ compare with $p_N(n)$  as $N$ and $n$ both go to $\infty$. The answer, perhaps surprisingly, is that when $N$ and $n$ grow at approximately the same rate, the first waves quickly become much larger than $p_N(n)$ (in absolute value, since these waves also oscillate like a sine with period $\approx 31.963$ in $N$).

The asymptotics of the first 100 waves is given in \cite{OSpsw} as follows, in terms of two uniquely defined complex numbers with approximations $w_0 \approx 0.916198 - 0.182459 i$ and  $z_0 \approx 1.181475 + 0.255528 i$.

\begin{theorem} \label{max2}
Let $\lambda^+$ be a positive real number.  Suppose $N \in \Z_{\gqs 1}$ and  $\lambda N \in \Z$ for $\lambda$ satisfying $|\lambda| \lqs \lambda^+$.
Then there are explicit coefficients $a_{0}(\lambda),$ $a_{1}(\lambda), \dots $ so that
\begin{equation} \label{pres}
   \sum_{k=1}^{100} W_k(N,\lambda N) = \Re\left[\frac{w_0^{-N}}{N^{2}} \left( a_{0}(\lambda)+\frac{a_{1}(\lambda)}{N}+ \dots +\frac{a_{m-1}(\lambda)}{N^{m-1}}\right)\right] + O\left(\frac{|w_0|^{-N}}{N^{m+2}}\right)
\end{equation}
as $N \to \infty$ where  $a_{0}(\lambda)=2  z_0 e^{-\pi i z_0(1+2\lambda)}$ and the implied constant depends only on  $\lambda^+$ and $m$.
\end{theorem}

In the rest of this section we briefly sketch the proof of Theorem  \ref{max2}, highlighting the role of Perron's method in the form of Corollary \ref{ilf}.
We require the dilogarithm, which is
initially defined as
\begin{equation*}
\li(z):=\sum_{n=1}^\infty \frac{z^n}{n^2} \quad \text{for} \quad |z|\lqs 1,
\end{equation*}
with an analytic continuation given by $ -\int_0^z \log(1-u)/u \, du$.

\begin{proof}[Sketch of proof of Theorem \ref{max2}]
In  \cite[Eq. (3.6)]{OSpsw}, it is shown that the left side of \eqref{pres} may be expressed as a sum of three parts. As in the proof of \cite[Thm. 1.2]{OSpsw}, two of these parts are $O(e^{0.055N})$. The third part may be expressed as an integral, see \cite[Eq. (5.13)]{OSpsw}, to obtain
\begin{equation} \label{ire}
    \sum_{k=1}^{100} W_k(N,\lambda N) = \frac{2}{N^{3/2}} \Im \int_{1.01}^{1.49} e^{N \cdot p(z)}  f_{\lambda}(z) \cdot \exp\bigl(v(z;N)\bigr) \, dz +O(e^{0.055N}),
\end{equation}
for an implied constant depending only on  $\lambda^+$, where
\begin{align*}
    p(z) & :=\frac{ \li\left(e^{2\pi i z}\right)-\li(1)}{2\pi i z}, \\
f_{\lambda}(z) & := \left( \frac{z }{2\sin(\pi(z -1))}\right)^{1/2} \exp\bigl(-\pi i z(2\lambda +1/2)\bigr).
\end{align*}
(In \cite{OSpsw}, the function $p(z)$ is used  with the opposite sign.)
To describe a useful approximation to the function $\exp\bigl(v(z;N)\bigr)$, we first define
\begin{align*}
    g_\ell(z) & :=-\frac{B_{2\ell}}{(2\ell)!} \left( \pi z\right)^{2\ell-1} \cot^{(2\ell-2)}\left(\pi z\right),\\
    u_{j}(z) & :=\sum_{m_1+3m_2+5m_3+ \dots =j}\frac{g_1(z)^{m_1}}{m_1!}\frac{g_2(z)^{m_2}}{m_2!} \cdots \frac{g_j(z)^{m_j}}{m_j!}
\end{align*}
with $u_{0}:=1$. Also define the box
\begin{equation*}\label{box}
    \mathbb B_1 := \{z\in \C \ : \ 1.01\lqs \Re(z) \lqs 1.49, \ -1 \lqs \Im(z) \lqs 1\}.
\end{equation*}
Then there are  functions $u_{j}(z)$ (defined above) and $\zeta_d(z;N)$ which are holomorphic on a domain containing the box $\mathbb B_1$ and have the following property. For all $z \in \mathbb B_1$,
\begin{equation} \label{cadiz}
    \exp\bigl(v(z;N)\bigr) = \sum_{j=0}^{d-1} \frac{u_{j}(z)}{N^j} + \zeta_d(z;N) \quad \text{for} \quad \zeta_d(z;N) = O\left(\frac{1}{N^d} \right)
\end{equation}
with an implied constant depending only on  $d$ where $1 \lqs d \lqs 2L-1$ and $L=\lfloor 0.006 \pi e \cdot N \rfloor$.

Since $|\exp(-2\pi i \lambda z )| \lqs \exp(\lambda^+ 2\pi   |z| )$ it follows that
\begin{equation}\label{soota}
    f_{\lambda}(z)  \ll 1 \quad \text{for} \quad z \in \mathbb B_1
\end{equation}
with an implied constant depending only on $\lambda^+$.

To apply Corollary \ref{ilf} we need the relevant saddle-point of $p(z)$ and this turns out to be $z_0:= 1+\log(1-w_0)/(2\pi i)$ where $w_0$ is the unique solution to $\li(w)-2\pi i \log(w) = 0$. Both $z_0$ and $w_0$ may be found to any precision and their approximations were given before Theorem  \ref{max2}. (It is straightforward to compute the size of the error introduced into \eqref{pres} by using approximations to $z_0$ and $w_0$.) We find $\mu=2$, $p_0\approx 0.504-0.241i$ and the steepest-descent angles are $\theta_0\approx 0.223$ and $\theta_1=\pi+\theta_0$.

Let $c:=1+i\Im(z_0)/\Re(z_0)$. We move the path of integration in \eqref{ire} to the path $\mathcal P$ through $z_0$ consisting of the straight line segments joining the points $1.01,$  $1.01c,$  $1.49c$ and $1.49$. Since the integrand in \eqref{ire} is holomorphic on a domain containing $\mathbb B_1$, Cauchy's theorem ensures that the integral remains the same under this change of path.
It is proved in \cite[Thm. 5.2]{OS1} that
\begin{equation}\label{aqw}
    \Re(p(z)-p(z_0))<0 \quad \text{for all} \quad z \in \mathcal P, \ z \neq z_0.
\end{equation}
We also need from \cite[Eq. (5.16)]{OSpsw} that
\begin{equation*}
 e^{p(z_0)} = w_0^{-1} \quad \text{and} \quad   e^{\Re(p(z_0))} = |w_0|^{-1} \approx e^{0.068}.
\end{equation*}

Using \eqref{cadiz} in \eqref{ire} implies
\begin{multline}\label{umand}
    \sum_{k=1}^{100} W_k(N,\lambda N)  = \Im\Biggl[ \sum_{j=0}^{d-1} \frac{2}{N^{3/2+j}}  \int_{\mathcal P} e^{N \cdot p(z)} \cdot f_{\lambda}(z) \cdot u_{j}(z) \, dz\\
     + \frac{2}{N^{3/2}}  \int_{\mathcal P} e^{N \cdot p(z)} \cdot f_{\lambda}(z) \cdot \zeta_d(z;N) \, dz \Biggr]+ O(e^{0.055N})
\end{multline}
where, by \eqref{cadiz}, \eqref{soota} and \eqref{aqw}, the last term in parentheses in \eqref{umand} is
\begin{equation*}
    \ll \frac{1}{N^{3/2}}  \int_{\mathcal P} \left|e^{N \cdot p(z)}\right| \cdot 1 \cdot \frac{1}{N^d} \, dz
    \ll \frac{1}{N^{d+3/2}} e^{N \Re(p(z_0))} = \frac{|w_0|^{-N}}{N^{d+3/2}},
\end{equation*}
for an implied constant depending only on $\lambda^+$.
Applying Corollary \ref{ilf} to each integral  in the first part of \eqref{umand}
 we obtain, since $k=0$,
\begin{equation} \label{wmand}
    \int_{\mathcal P} e^{N \cdot p(z)} \cdot f_{\lambda}(z) \cdot u_{j}(z) \, dz =  e^{N \cdot p(z_0)}\left(\sum_{m=0}^{M-1}\G\left(m+\frac 12 \right) \frac{2\alpha_{2m}(f_{\lambda} \cdot u_{j})}{N^{m+1/2}}+O\left( \frac{K(f_{\lambda} \cdot u_{j})}{N^{M+1/2}}\right) \right).
\end{equation}
We have written $\alpha_{2m}(q)$, to show the dependence of  $\alpha_{2m}$ on $q=f_{\lambda} \cdot u_{j}$, and also $K(q)$ instead of $K_q$.
The error term in \eqref{wmand} corresponds to an error in \eqref{umand} of size $O(|w_0|^{-N}/N^{M+j+2})$.
Choose $M=d$ so that this error  is less than $O(|w_0|^{-N}/N^{d+3/2})$ for all $j \gqs 0$.
Therefore
\begin{align}
    \sum_{k=1}^{100} W_k(N,\lambda N) & = \Im \left[
    \sum_{j=0}^{d-1} \frac{4}{N^{j+3/2}}   e^{N \cdot p(z_0)} \sum_{m=0}^{d-1} \G\left(m+\frac 12 \right) \frac{\alpha_{2m}(f_{\lambda} \cdot u_{j})}{N^{m+1/2}}
    \right]+ O\left( \frac{|w_0|^{-N}}{N^{d+3/2}}\right) \notag\\
    & = \Im \left[  w_0^{-N}
    \sum_{t=0}^{2d-2} \frac{4}{N^{t+2}}    \sum_{m=\max(0,t-d+1)}^{\min(t,d-1)} \G\left(m+\frac 12 \right)  \alpha_{2m}(f_{\lambda} \cdot u_{t-m})
    \right]+ O\left( \frac{|w_0|^{-N}}{N^{d+3/2}}\right) \label{jmp}\\
    & = \Re \left[  w_0^{-N}
    \sum_{t=0}^{d-2} \frac{-4i}{N^{t+2}}    \sum_{m=0}^{t} \G\left(m+\frac 12 \right)  \alpha_{2m}(f_{\lambda} \cdot u_{ t-m})
    \right]+ O\left( \frac{|w_0|^{-N}}{N^{d+1}}\right) \label{jmp2}
\end{align}
for implied constants depending only on $\lambda^+$ and $d$. (In going from \eqref{jmp} to \eqref{jmp2} we used that $|\alpha_{2m}(f_{\lambda} \cdot u_{j})|$ has a bound  depending only  on  $\lambda^+$ and $d$, by Proposition \ref{albnd}, when $m,$ $j \lqs d-1$.)
 Hence,  with
\begin{equation} \label{btys}
    a_t(\lambda):=  -4i  \sum_{m=0}^t \G\left(m+\frac 12 \right) \alpha_{2m}(f_{\lambda} \cdot u_{t-m}),
\end{equation}
we obtain \eqref{pres} in the statement of the theorem.

The first coefficient is
\begin{equation} \label{oad}
    a_0(\lambda)=-4i \G(1/2) \alpha_0(f_{\lambda} \cdot u_{0,0})= -4i \sqrt{\pi} \alpha_0(f_{\lambda})= -2i \sqrt{\pi} p_0^{1/2} f_{\lambda}(z_0),
\end{equation}
using our formula for $\alpha_0$ from Section \ref{formul}. The calculations  \cite[Eqs. (5.24), (5.26)]{OSpsw} show
\begin{equation} \label{woad2}
    p_0^{1/2}  = -\frac{\sqrt{\pi} e^{-\pi i/4} e^{\pi i z_0}}{z_0^{1/2} w_0^{1/2}},\qquad
    f_{\lambda}(z_0)   = -\frac{e^{\pi i/4} z_0^{1/2}}{w_0^{1/2}}  e^{-2\pi i \lambda z_0}.
\end{equation}
The formula for $a_0(\lambda)$ in the statement of the theorem follows from \eqref{oad} and \eqref{woad2}.
\end{proof}

We may take $N=2000$ and $\lambda=1$ as an example of Theorem \ref{max2}.  The first wave $W_1(N,N)$ is $\approx 4.37 \times 10^{53}$ with the next waves much smaller: $W_2(N,N) \approx 4.98 \times 10^{23},$  $W_3(N,N) \approx -8.22 \times 10^{13}$ etc. We find that the main term on the right of \eqref{pres} is $\approx 4.56 \times 10^{53}$. Taking the first $3$ terms on the right of \eqref{pres} gives the more accurate $4.37 \times 10^{53}$. By comparison, the corresponding partition number $p(N)$ ($=p_N(N)$) is a lot smaller and approximately $4.72 \times 10^{45}$.

See \cite{OSpsw} for the detailed proof of Theorem \ref{max2} as well as more extensive discussion and numerical work. We expect, as in \cite[Conjecture 9.1]{OSpsw}, that Theorem \ref{max2} is true with the sum of the first 100 waves on the left of \eqref{pres}  replaced by just the first wave $ W_1(N,\lambda N)$.

{\small
\bibliography{asymptotics}
}

{\small 
\vskip 5mm
\noindent
\textsc{Dept. of Math, The CUNY Graduate Center, 365 Fifth Avenue, New York, NY 10016-4309, U.S.A.}

\noindent
{\em E-mail address:} \texttt{cosullivan@gc.cuny.edu}
}

\end{document}